\documentclass{amsart}

\usepackage[all]{xypic}
\usepackage{tikz}
\usetikzlibrary{arrows}
\usepackage{epsf}
\usepackage{verbatim} 
\usepackage{amsmath}
\usepackage{amsfonts}
\usepackage{amssymb}
\usepackage{mathrsfs}
\usepackage{amsthm}
\usepackage{newlfont}

 \newtheorem{thm}{Theorem}
 
 \newtheorem{cor}[thm]{Corollary}

 \newtheorem{lem}[thm]{Lemma}

 \theoremstyle{definition}
 \newtheorem{defn}[thm]{Definition}
 
 \theoremstyle{definition}
 \newtheorem{notn}[thm]{Notation}

 \theoremstyle{remark}
 \newtheorem{rem}[thm]{Remark}
 
 \theoremstyle{definition}
 \newtheorem{example}[thm]{Example}

 \numberwithin{thm}{section}
 \numberwithin{equation}{section}

 \newcommand{\Hom}{\mathrm{Hom}}
 \newcommand{\Spec}{\mathrm{Spec}}

 \newcommand{\End}{\mathrm{End}}
 
 \newcommand{\Pic}{\mathrm{Pic}}

 \newcommand{\Vol}{\mathrm{Vol}}

 \newcommand{\GL}{\mathrm{GL}}
 \newcommand{\PGL}{\mathrm{PGL}}
 
 \newcommand{\SL}{\mathrm{SL}}

 \newcommand{\rank}{\mathrm{rank}}

 \newcommand{\Tr}{\mathrm{Tr}}

 \renewcommand{\mod}{\mathrm{mod}}
\newcommand{\sT}{\mathscr{T}}

 \newcommand{\fp}{\mathfrak p}
 
 \newcommand{\fn}{\mathfrak n}
 \newcommand{\fl}{\mathfrak l}
 \newcommand{\fm}{\mathfrak m}

 \newcommand{\fE}{\mathfrak E}
 \newcommand{\fD}{\mathfrak D}
 
 \newcommand{\cO}{\mathcal{O}}

 \newcommand{\cG}{\mathcal{G}}

 \newcommand{\cX}{\mathcal{X}}

 \newcommand{\cJ}{\mathcal{J}}
 \newcommand{\cI}{\mathcal{I}}
 
 \renewcommand{\cH}{\mathcal{H}}

 \newcommand{\R}{\mathbb{R}}
 
 \newcommand{\C}{\mathbb{C}}
 \newcommand{\F}{\mathbb{F}}
 \newcommand{\Q}{\mathbb{Q}}
 \newcommand{\Z}{\mathbb{Z}}

 \newcommand{\p}{\mathbb{P}}
 
 \newcommand{\T}{\mathbb{T}}

 \newcommand{\eps}{\varepsilon}
 \newcommand{\To}{\longrightarrow}
 \newcommand{\bs}{\setminus}
 
 \newcommand{\Fi}{F_\infty}

 \newcommand{\G}{\Gamma}
 \newcommand{\La}{\Lambda}
 \newcommand{\la}{\lambda}
 \newcommand{\norm}[1]{|\!|#1|\!|}

\begin{document}

\title[Graph Laplacians and Drinfeld modular curves]{Graph Laplacians, component groups and \\ Drinfeld modular curves}

\author{Mihran Papikian}
\address{Department of Mathematics, Pennsylvania State University, University Park, PA 16802, U.S.A.}
\email{papikian@psu.edu}

\thanks{The author's research was partially supported by grants from the Simons Foundation (245676) and the National Security Agency 
(H98230-15-1-0008).} 

\subjclass[2010]{11G18, 05C25,  11F12} 
\keywords{Graph Laplacian; Component group; Drinfeld modular curve; Hecke algebra; Eisenstein ideal.}



\begin{abstract} Let $\fp$ be a prime ideal of $\F_q[T]$. Let $J_0(\fp)$ 
be the Jacobian variety of the Drinfeld modular curve $X_0(\fp)$. Let $\Phi$ 
be the component group of $J_0(\fp)$ at the place $1/T$. We use graph Laplacians to 
estimate the order of $\Phi$ as $\deg(\fp)$ goes to infinity.  
This estimate implies that $\Phi$ is not annihilated by the Eisenstein ideal of the Hecke algebra 
$\T(\fp)$ acting on $J_0(\fp)$ once the degree of $\fp$ is large enough. We also obtain an asymptotic formula for the size of the discriminant 
of $\T(\fp)$ by relating this discriminant to the order of $\Phi$; in this problem the order of $\Phi$
plays a role similar to the Faltings height of classical modular Jacobians. 
Finally, we bound the spectrum of the adjacency operator of a finite subgraph of an infinite  
diagram in terms of the spectrum of the adjacency operator of the diagram itself; this 
result has applications to the gonality of Drinfeld modular curves. 
\end{abstract}


\maketitle

\tableofcontents


\section{Introduction}  

Let $\F_q$ be a finite field with $q$ 
elements, where $q$ is a power of a prime number $p$. Let $A=\F_q[T]$ 
be the ring of polynomials in indeterminate $T$ with coefficients 
in $\F_q$, and $F=\F_q(T)$ be the rational function field. 
The degree map $\deg: F\to \Z\cup \{-\infty\}$, which assigns 
to a non-zero polynomial its degree in $T$ and $\deg(0)=-\infty$, defines 
a norm on $F$ by $|a|:=q^{\deg(a)}$. The corresponding place of $F$ 
is usually called the \textit{place at infinity}, and is denoted by $\infty$. 
Note that $1/T$ is a uniformizer at $\infty$. 
The order of a finite set $S$ will be denoted by $|S|$. 
We define norm and degree for a non-zero ideal $\fn$ of $A$ by $|\fn|:=|A/\fn|$ and $\deg(\fn):=\log_q|\fn|$. 
The prime ideals $\fp\lhd A$ always will be assumed to be non-zero. 

Let $\Fi$ be the completion of $F$ at $\infty$, and $\C_\infty$ be the completion of an algebraic closure of $\Fi$. 
The \textit{Drinfeld upper half-plane} $\Omega:=\C_\infty-\Fi$
has a natural structure of a rigid-analytic space over $\Fi$; cf. \cite{Drinfeld}, \cite{GR}. 
Let $\fn\lhd A$ be a non-zero ideal. The level-$\fn$ \textit{Hecke congruence subgroup} of $\GL_2(A)$ is 
$$
\G_0(\fn):=\left\{\begin{pmatrix} a & b \\  c & d\end{pmatrix}\in \GL_2(A)\ \bigg|\ c\equiv 0\ \mod\  \fn \right\}.   
$$
The group $\G_0(\fn)$ acts on $\Omega$ via linear fractional transformations. 
Drinfeld proved that the quotient $\G_0(\fn)\bs \Omega$ 
is the space of $\C_\infty$-points of an affine curve $Y_0(\fn)$ defined over $F$,  
which is a coarse moduli scheme for rank-$2$ Drinfeld $A$-modules with $\G_0(\fn)$-level structures; cf. \cite{Drinfeld}, \cite{GR}.
Let $X_0(\fn)$ be the unique smooth projective curve over $F$ containing $Y_0(\fn)$. 
The curve $X_0(\fn)$ is geometrically irreducible. 
Let $J_0(\fn)$ be the Jacobian variety of $X_0(\fn)$. 

The analogy between $X_0(\fn)$ 
and the classical modular curves $X_0(N)$ over $\Q$ classifying elliptic curves with 
$\G_0(N)$-structures is well-known and has been extensively studied over the last 40 years. 
From this perspective $\infty$ plays a role similar to the archimedean place of $\Q$, and 
$\Omega$ plays the role of the Poincar\'e upper half-plane. In this paper we study a certain group 
associated to $J_0(\fn)$, the component group at $\infty$, for which there is no direct classical analogue. 

For a place $v$ of $F$, let $\Phi_{J_0(\fn), v}$ denote the 
group of connected components of the N\'eron model of $J_0(\fn)$ at $v$. 
Apart from $\infty$, the places of $F$ are in bijection with the non-zero prime ideals of $A$. 
It is known that $J_0(\fn)$ has bad reduction only at $v$ dividing $\fn$ and at $\infty$, so $\Phi_{J_0(\fn), v}$ is non-trivial 
only if $v|\fn$ or $v=\infty$. By a theorem of Raynaud, the group structure of $\Phi_{J_0(\fn), v}$ can be 
deduced from the structure of the special fibre of the minimal regular model of $X_0(\fn)$ over $v$. 
If $v\neq \infty$, the structure of the minimal regular model 
itself can be deduced from the moduli interpretation of $X_0(\fn)$. 
For example, if $v\parallel \fn$ (i.e., $v$ divides $\fn$ but $\fn/v$ is coprime to $v$), 
the structure of $\Phi_{J_0(\fn), v}$ as an abelian group is given in \cite[Thm. 5.3]{PW}; see also \cite{Uber}. 
One consequence of this description is that for $v\parallel \fn$ the order of $\Phi_{J_0(\fn), v}$ grows linearly with $|\fn|$. 
For example, if $\fp\lhd A$ is prime, then $\Phi_{J_0(\fp), \fp}$ is a cyclic 
group of order 
$$
N(\fp)=\begin{cases}
\frac{|\fp|-1}{q-1}, & \text{if $\deg(\fp)$ is odd};\\
\frac{|\fp|-1}{q^2-1}, & \text{if $\deg(\fp)$ is even.}
\end{cases}
$$

In contrast, the group $\Phi_{J_0(\fn), \infty}$ seems to be a much more complicated object, 
and no general formulas for its order are known (even for prime $\fn$). 
The central result of this paper is an estimate on the order of $\Phi_{J_0(\fn), \infty}$.  

\begin{notn}
Let $f(x)$ and $g(x)$ be positive real valued functions defined on $\Z_{>0}$, or ideals of 
$A$, or prime ideals of $A$. We write $f(x)=O(g(x))$ when there is a constant $C$ such that 
$f(x)\leq C g(x)$ for all values of $x$ under consideration.  
We write $f(x)\sim g(x)$ when $\lim_{|x|\to \infty}f(x)/g(x)= 1$, and $f(x)=o(g(x))$ when $\lim_{|x|\to \infty}f(x)/g(x)= 0$. 
\end{notn}

\begin{thm}\label{thmMain1} 
Let $\fp\lhd A$ be a prime ideal. We have
$$
\ln \left|\Phi_{J_0(\fp), \infty}\right|  \sim c(q)|\fp|,
$$
where $c(q)$ is an explicit positive constant depending only on $q$. This constant can be estimated as 
$$
c(q)=\frac{2 \ln\left(q+\frac{1}{2}\right)}{(q-1)^2(q+1)} + O(q^{-5}\ln q). 
$$
\end{thm}

The restriction on $\fp$ being prime is made 
only for expository reasons. In fact, the methods that we develop for proving this theorem 
apply to any congruence subgroup $\G$ of $\GL_2(A)$, and show that the order of the component group at $\infty$ 
of the corresponding Drinfeld modular Jacobian can be estimated in a similar manner with $|\fp|$ replaced by $[\GL_2(A):\G]$. 
In particular, the orders of component groups grow exponentially with $[\GL_2(A):\G]$. 

\vspace{0.1in}

Theorem \ref{thmMain1} has two interesting applications to the Hecke algebra acting on $J_0(\fp)$. 
Let $\T(\fn)\subset \End(J_0(\fn))$ 
be the $\Z$-subalgebra of the endomorphisms of $J_0(\fn)$ generated by the Hecke operators $T_\fm$,  
$\fm\lhd A$, acting as correspondences on $X_0(\fn)$. 
The \textit{Eisentein ideal} $\fE(\fn)$ of $\T(\fn)$ is the ideal generated by the elements 
$$
\left\{T_\fl-|\fl|-1\ \big|\ \fl \text{ is prime}, \fl\nmid \fn\right\}.
$$
 
 It is well-known that the component groups of classical modular Jacobians $J_0(N)$ are Eisenstein, i.e., 
 are annihilated by $T_\ell-\ell-1$ for all prime $\ell$ not dividing $N$. 
This was proved by Ribet in the semistable reduction case \cite{RibetCGSS}, and 
by Edixhoven in general \cite{EdixhovenECG}. 
It is more-or-less clear that the arguments in \cite{RibetCGSS} and 
\cite{EdixhovenECG} can be transferred to the function fields setting (although this is not in published literature), so 
it is very likely that the component groups of Drinfeld modular Jacobians $J_0(\fn)$ at $v\mid \fn$ are Eisenstein. 
In any case, the fact that $\Phi_{J_0(\fp), \fp}$ is Eisenstein follows from the results in \cite{Uber}. In \cite[Thm. 8.9]{PW2}, 
it is shown that the $\T(\fp)$-submodule of $\Phi_{J_0(\fp), \infty}$ annihilated by $\fE(\fp)$ is isomorphic to $\T(\fp)/\fE(\fp)\cong \Z/N(\fp)\Z$. 
Comparing this with the estimate in Theorem \ref{thmMain1}, we see that 
\begin{thm} The component group 
$\Phi_{J_0(\fp), \infty}$ is not Eisenstein if $\deg(\fp)$ is large enough. 
\end{thm}
\begin{rem} 
Interestingly, even the groups of connected components 
of the real points $J_0(N)(\R)$ of classical modular Jacobians are Eisenstein, as was shown by Merel \cite{Merel}. 
\end{rem}

Let $N$ be a square-free integer. The discriminant $\fD_{\T(N)}$ of the Hecke algebra $\T(N)$ acting on the 
classical modular Jacobian $J_0(N)$ measures congruences between weight-$2$ cusp forms 
on $\G_0(N)$. In \cite{Ullmo}, Ullmo obtained the following bounds:
\begin{equation}\label{UllmoBound}
g(N)\ln N + o(g(N)\ln N) \leq  \ln \fD_{\T(N)} \leq 2g(N)\ln N+o(g(N)\ln N),
\end{equation}
where $g(N)$ is the genus of $X_0(N)$. To prove this he first showed that $\fD_{\T(N)}$ is related 
to the Faltings height of $J_0(N)$. The lower bound in \eqref{UllmoBound} then follows from 
a general lower bound on the heights of abelian varieties over number fields due to Bost. In the reverse direction, 
the upper bound on $\fD_{\T(N)}$ gives an upper bound on the height of $J_0(N)$. 

Now let $\fp\lhd A$ be a prime ideal. Denote by $g(\fp)$ the genus of $X_0(\fp)$. It is known that $g(\fp)\sim |\fp|/(q^2-1)$; 
see $\S$\ref{ssDMC} for an explicit formula. Let $H(J_0(\fp))$ be the height of $J_0(\fp)$; see 
$\S$\ref{ssJRC} for the definition. Let $\fD_{\T(\fp)}$ be the discriminant of the Hecke algebra $\T(\fp)$; see \eqref{deffDT} 
for the definition. Using the results of Szpiro \cite{Szpiro}, 
it is not particularly hard to prove the following bounds on the height (Theorem \ref{lastTheorem})
$$
\frac{g(\fp)\deg(\fp)}{12}+o(g(\fp)\deg(\fp))\leq  H(J_0(\fp))\leq \frac{g(\fp)^2\deg(\fp)}{3}+o(g(\fp)^2\deg(\fp)). 
$$
On the other hand, the discriminant $\fD_{\T(\fp)}$ does not seem 
to be directly related to $H(J_0(\fp))$; the height is defined in terms of differential forms on $J_0(\fp)$, 
which correspond to $\C_\infty$-valued Drinfeld modular forms, whereas $\fD_{\T(\fp)}$ measures 
congruences between $\C$-valued automorphic forms on $\G_0(\fp)$. Nevertheless, we show 
that a crucial part of Ullmo's argument does go through with $\left|\Phi_{J_0(\fp), \infty}\right|$ playing the role 
of the height. This gives a formula relating $\left|\Phi_{J_0(\fp), \infty}\right|$ and $\fD_{\T(\fp)}$; see Theorem \ref{thmfDT}. Using this formula and 
Theorem \ref{thmMain1}, we obtain in $\S$\ref{ssDMC} the following: 

\begin{thm}\label{thmMain2} Let $\fp\lhd A$ be prime. Then 
$$
2g(\fp)\deg(\fp)+o(g(\fp)\deg(\fp))\leq \log_q(\fD_{\T(\fp)}). 
$$
If a certain natural pairing \eqref{GPairing} between $\T(\fp)$ and the space of $\Z$-valued $\G_0(\fp)$-invariant 
harmonic cochains is perfect, then 
$$
\log_q(\fD_{\T(\fp)})\sim 2g(\fp)\deg(\fp). 
$$
\end{thm}

\vspace{0.1in}

To prove Theorem \ref{thmMain1} we relate the order of $\Phi_{J_0(\fp), \infty}$ to the 
eigenvalues of a certain Hecke operator, and then use some deep facts about these eigenvalues,  
such as the Ramanujan-Petersson estimate on their absolute values and their equidistribution 
with respect to a certain Sato-Tate measure. 
To relate $\Phi_{J_0(\fp), \infty}$ to a Hecke operator, in Section \ref{secDiagrams}, we prove 
two general combinatorial results of independent interest. 

The first combinatorial result (Theorem \ref{thmMF})
relates the discriminant of the weighted cycle pairing on the first homology group of a graph 
to the eigenvalues of the weighted Laplacian on the graph. We allow both the vertices and the edges 
of the graph to have weights. When all the weights are equal to $1$, our theorem specializes 
to a result of Lorenzini \cite{Lorenzini}. The reason that we need to work with weighted graphs is that the graph that arises 
in our context is the quotient of the Bruhat-Tits tree $\sT$ of $\PGL_2(\Fi)$ under the action of $\G_0(\fp)$. 
The graph $\G_0(\fp)\bs \sT$ is naturally weighted, with the weighted adjacency operator 
corresponding to a Hecke operator. 
The arithmetic application of Theorem \ref{thmMF} 
is that it relates the order of the component group of the Jacobian of a semi-stable, but 
not necessarily regular, curve over a local domain to the eigenvalues of a weighted Laplacian acting 
on its dual graph.  

The second result (Theorem \ref{thmWeylHilb}) concerns certain 
infinite graphs, called regular diagrams. We bound the spectrum of the adjacency 
operator of a finite subgraph of a diagram in terms of the spectrum of the adjacency operator of the diagram 
itself. The arithmetic application of Theorem \ref{thmWeylHilb} is that, when combined with the Ramanujan-Petersson  
conjecture, it implies that the minimal non-zero 
eigenvalue $\la_2$ of the Laplacian of the dual graph of $X_0(\fn)$ over $\infty$ is bounded from below by $q-2\sqrt{q}$; see $\S$\ref{ssEE}. 
This bound on $\la_2$ plays an important role in \cite{CKK}. 

\begin{rem} A proof of the bound $\la_2\geq q-2\sqrt{q}$ already appears 
in \cite[pp. 245-246]{CKK}. Unfortunately, that proof is not correct. The problem is that the spectrum of a finite subgraph 
of a diagram is not necessarily contained in the discrete spectrum of the diagram itself. In particular, 
the function $\tilde{f}$ constructed on page 245 of \cite{CKK} is not necessarily square-integrable,  
hence is not an automorphic form. For a more detailed discussion of this see $\S$\ref{ssRD}. 
\end{rem}

\subsection*{Acknowledgements} I thank Robert Vaughan for providing the proof of Lemma \ref{lemBobV}, 
and Dale Brownawell and Winnie Li for useful conversations. I also thank the anonymous referee 
for her/his careful reading of an earlier version of this article and numerous helpful remarks.


\section{Preliminaries} 

\subsection{Graphs and Laplacians}\label{sec1} 
A \textit{graph} consists of a set of vertices $V(G)$, a set of (oriented) edges $E(G)$ 
and two maps 
$$
E(G)\to V(G)\times V(G), \quad 
e\mapsto (o(e), t(e))
$$
and 
$$
E(G)\to E(G), \quad e\mapsto \bar{e}
$$ 
such that $\bar{\bar{e}}=e$, $\bar{e}\neq e$, and $t(\bar{e})=o(e)$; cf. \cite[p. 13]{SerreT}. 
 
For $e\in E(G)$, the edge $\bar{e}$  is called the \textit{inverse} of $e$, the vertex $o(e)$ (resp. $t(e)$) is called   
the \textit{origin} (resp. \textit{terminus}) of $e$. The vertices $o(e), t(e)$ are called the \textit{extremities} 
(or \textit{end-vertices}) of $e$. We say that two vertices are \textit{adjacent} if they are the extremities of some edge. 
An \textit{orientation} of $G$ is a subset $E(G)^+$ of $E(G)$ such that $E(G)$ is the disjoint union of 
$E(G)^+$ and $\overline{E(G)^+}$. 

A \textit{path} in $G$ is a 
sequence of edges $\{e_i\}_{i\in I}$ indexed by a set $I$ where $I=\Z$, $I=\Z_{\geq 0}$ or 
$I=\{1,\dots, m\}$ for some $m\geq 1$ such that $t(e_i)=o(e_{i+1})$ for every 
$i, i+1\in I$.  We say that the path is \textit{without backtracking} if $e_i\neq \bar{e}_{i+1}$ 
for every $i, i+1\in I$. We say that the path without backtracking $\{e_i\}_{i\in \Z_{\geq 0}}$ is a \textit{half-line} 
if $o(e_i)$ is adjacent in $G$ only to $o(e_{i-1})$ and $t(e_i)$, $i\geq 1$. 

We will assume that for any $v\in V(G)$ the number of edges with $t(e)=v$ is finite,  
and that $G$ is connected, i.e., any two vertices of $G$ are connected by a path. 
In addition, we assume that $G$ has no loops (i.e., $t(e)\neq o(e)$ for any $e\in E(G)$),  
but we allow two vertices to be joined by multiple edges (i.e., there can be $e\neq e'$ 
with $o(e)=o(e')$ and $t(e)=t(e')$).  We say that $G$ is \textit{finite} if it has finitely many vertices. 

Since $G$ has no loops, we can consider $G$ as a simplicial complex. 
Choose an orientation $E(G)^+$ on $G$, and define the group $C_i(G, \Z)$ of $i$-dimensional chains of $G$ ($i=0,1$) by 
\begin{align*}
C_0(G, \Z) &=\text{free abelian group with basis }V(G),\\
C_1(G, \Z) &=\text{free abelian group with basis }E(G)^+. 
\end{align*}
(One can also define $C_1(G, \Z)$ as the quotient of the free abelian group with basis $E(G)$ 
modulo the relations $\bar{e}=-e$.)  Since $G$ is not assumed to be finite, it might be worth spelling out that 
a general element of $C_0(G, \Z)$ has the form $\sum_{v\in V(G)}n_v v$, $n_v\in \Z$, where all 
but finitely many of $n_v$ are zero (and similarly for $C_1(G, \Z)$). We have the homomorphisms 
\begin{align*}
&\partial: C_1(G, \Z)\to C_0(G, \Z) \quad \text{ given by } \quad \partial(e)=t(e)-o(e), \\
& \eps: C_0(G, \Z)\to \Z \quad \text{ given by }\quad \eps(v)=1.   
\end{align*}
Let $H_1(G, \Z):=\ker(\partial)$ be the first homology group of $G$. Then there is an exact sequence 
\begin{equation}\label{eq1}
\xymatrix{0 \ar[r] & H_1(G, \Z)  \ar[r] & C_1(G, \Z) \ar[r]^-{\partial} & C_0(G, \Z) \ar[r]^-{\eps} & \Z \ar[r] & 0.}
\end{equation}

A \textit{weight function} on edges is a map $w: E(G)\to \Z_{>0}$ such that $w(e)=w(\bar{e})$. Define a pairing  
$E(G) \times E(G) \to \Z$: 
\begin{equation}\label{eqweighte}
( e, e') = 
\begin{cases} 
w(e) & \text{if $e'=e$},\\
-w(e) & \text{if $e'=\bar{e}$},\\
0 & \text{otherwise}, 
\end{cases}
\end{equation}
and extend it linearly to a symmetric, bilinear, positive-definite pairing on $C_1(G, \Z)$. The restriction 
of this pairing to $H_1(G, \Z)$ is a weighted version of the usual cycle pairing.  

A \textit{weight function} on vertices is a map $w: V(G)\to \Z_{>0}$. Define  
a pairing $V(G) \times V(G)\to \Z $:
\begin{equation}\label{pairingonV}
\langle v, v'\rangle=
\begin{cases}
w(v) & \text{if $v=v'$}, \\
 0 & \text{otherwise},
\end{cases}
\end{equation}
and extend it linearly to a symmetric, bilinear, positive-define pairing on $C_0(G, \Z)$. 
Given a $\Z$-module $R$, the previous two pairings naturally extend to 
$$
C_i(G, R): =C_i(G, \Z)\otimes_\Z R,
$$
and so does the boundary operator $\partial: C_1(G, R)\to C_0(G, R)$. 

Let $$\partial^\ast: C_0(G, \Q) \to C_1(G, \Q)$$ be the adjoint of $\partial$ with respect to the pairings \eqref{eqweighte} and \eqref{pairingonV}, i.e., 
$$
\langle\partial f, g \rangle =(f, \partial^\ast g) \quad \text{for all $f\in C_1(G, \Q)$ and $g\in C_0(G, \Q)$}. 
$$
It is easy to check that for a given vertex $v\in V(G)$ 
$$
\partial^\ast(v)=\sum_{t(e)=v}\frac{w(v)}{w(e)}e. 
$$

\begin{defn}\label{def1} The (weighted) \textit{Laplacian} is the composition 
$$\Delta=\partial\partial^\ast: C_0(G, \Q)\to C_0(G, \Q).$$ Explicitly, this map is given by 
$$
\Delta(v)=\sum_{t(e)=v}\frac{w(v)}{w(e)}(v-o(e)). 
$$
\end{defn}

For any $f, g\in C_0(G, \R)$ we have 
$$
\langle \Delta f, g\rangle=\langle\partial\partial^\ast f, g\rangle=(\partial^\ast f, \partial^\ast g)
=\langle f, \partial\partial^\ast g\rangle =\langle f, \Delta g\rangle
$$
and 
$$
\langle \Delta f, f\rangle = (\partial^\ast f, \partial^\ast f)\geq 0. 
$$
Thus, the linear operator $\Delta$ on $C_0(G, \R)$ is self-adjoint and positive. For finite $G$, this implies  
that $C_0(G, \R)$ has an orthonormal basis consisting of eigenvectors of $\Delta$, and the eigenvalues 
of $\Delta$ are nonnegative. In that case, it is also easy to show that the 
kernel of $\Delta$ is spanned by $f_0=\sum_{v\in V(G)} v/w(v)$, so $0$ is an eigenvalue of $\Delta$ 
with multiplicity one.

\begin{defn}\label{defnDisc} Assume $h=\rank_\Z H_1(G, \Z)$ is finite. 
Choose a $\Z$-basis $\varphi_1, \dots, \varphi_h$ 
of $H_1(G, \Z)$, and let $$\fD_{G,w}:=|\det((\varphi_i, \varphi_j))_{1\leq i,j\leq h}|.$$  
We call $\fD_{G,w}$ the \textit{discriminant} of $G$ with respect to the weight function $w$ in \eqref{eqweighte}; cf. \cite[p. 49]{SerreLF}.  
\end{defn}

\begin{lem}\label{lem4}
$\fD_{G,w}$ is the order of the cokernel of the map 
\begin{align*}
H_1(G, \Z) &\To \Hom(H_1(G, \Z), \Z) \\
\varphi & \mapsto ( \varphi, \ast ). 
\end{align*}
In particular, $\fD_{G,w}$ does not depend on the choice of a basis of $H_1(G, \Z)$. 
\end{lem}
\begin{proof}
This follows from Proposition 4 in $\S$III.2 of \cite{SerreLF}. 
\end{proof}


\subsection{Harmonic cochains}\label{ssHCHO}
Fix a commutative ring $R$ with identity. An $R$-valued \textit{harmonic cochain} on a graph $G$ is a 
function $f: E(G)\to R$ that satisfies 
\begin{itemize}
\item[(i)] $$f(e)+f(\bar{e})=0\quad \text{for all $e\in E(G)$},$$
\item[(ii)] 
$$\sum_{\substack{e\in E(G) \\ t(e)=v}} f(e)=0\quad \text{for all $v\in V(G)$}.$$
\end{itemize}
Denote by $\cH(G, R)$ the group of $R$-valued harmonic cochains on $G$. 

The most important graphs in this paper are the Bruhat-Tits tree $\sT$ of $\PGL_2(\Fi)$, and the 
quotients of $\sT$. We recall the definition and introduce some notation for later use; see \cite{SerreT} for more details. 
Fix a uniformizer $\varpi_\infty$ of $\Fi$, and let $\cO_\infty$ be its ring of integers. 
The sets of vertices $V(\sT)$ and edges $E(\sT)$ are the cosets $\GL_2(\Fi)/Z(\Fi)\GL_2(\cO_\infty)$ 
and $\GL_2(\Fi)/Z(\Fi)\cI_\infty$, respectively, where $Z$ denotes the center of $\GL_2$ and $\cI_\infty$ is the Iwahori group:
$$
\cI_\infty=\left\{\begin{pmatrix} a & b\\ c & d\end{pmatrix}\in \GL_2(\cO_\infty)\ \bigg|\ c\in \varpi_\infty\cO_\infty\right\}. 
$$
The matrix $\begin{pmatrix} 0 & 1\\ \varpi_\infty & 0\end{pmatrix}$ 
normalizes $\cI_\infty$, so the multiplication from the right by this matrix on $\GL_2(\Fi)$ 
induces an involution on $E(\sT)$; this involution is $e\mapsto \bar{e}$. 
The matrices 
\begin{equation}\label{eq-setM}
E(\sT)^+=\left\{\begin{pmatrix} \varpi_\infty^k & u \\ 0 & 1\end{pmatrix}\ \bigg|\
\begin{matrix} k\in \Z\\ u\in \Fi,\ u\ \mod\ \varpi_\infty^k\cO_\infty\end{matrix}\right\}
\end{equation}
are in distinct left cosets of $\cI_\infty Z(\Fi)$, and there is a disjoint decomposition 
$$
E(\sT)=E(\sT)^+\bigsqcup E(\sT)^+\begin{pmatrix} 0 & 1\\ \varpi_\infty & 0\end{pmatrix}. 
$$
We call the edges in $E(\sT)^+$ \textit{positively oriented}. 

The group $\GL_2(\Fi)$ naturally acts on $E(\sT)$ by left multiplication. 
This induces an action on the group of $R$-valued functions on $E(\sT)$: 
for a function $f$ on $E(\sT)$ and $\gamma\in \GL_2(\Fi)$ we define the function $f|\gamma$ on $E(\sT)$ by 
$(f|\gamma)(e)=f(\gamma e)$. 
It is clear from the definition that $f|\gamma$ is harmonic if $f$ is harmonic, and 
for any $\gamma, \sigma\in \GL_2(\Fi)$ we have $(f|\gamma)|\sigma=f|(\gamma\sigma)$. 

A \textit{congruence subgroup} is a subgroup $\G\leq \GL_2(A)$ containing 
$$
\G(\fn):=\left\{\gamma\in \GL_2(A)\ |\ \gamma\equiv \begin{pmatrix} 1 & 0\\ 0 & 1\end{pmatrix}\ \mod\ \fn\right\} 
$$
for some non-zero $\fn\lhd A$. A congruence subgroup $\G$, being a subgroup of $\GL_2(\Fi)$, acts on $\sT$.  
This action is without inversions, 
i.e., $\gamma e\neq \bar{e}$, $\forall \gamma\in \G$, $\forall e\in E(\sT)$; see \cite[p. 75]{SerreT}. 
We have a natural quotient graph
$\G\bs \sT$ such that $V(\G\bs \sT)=\G\bs V(\sT)$ and
$E(\G\bs \sT)=\G\bs E(\sT)$, cf. \cite[p. 25]{SerreT}. 
Given $v\in V(\sT)$ and $e\in E(\sT)$, let 
$$\G_v=\{\gamma\in \G\ |\ \gamma v=v\} \quad \text{and}\quad \G_e=\{\gamma\in \G\ |\ \gamma e=e\}. $$
Since $\G$ is a discrete subgroup of $\GL_2(\Fi)$, the groups $\G_v$ and $\G_e$ are finite. 
It is immediate from the definitions that $Z(\F_q)\cap \G$ is a normal subgroup of any $\G_v$ and $\G_e$. 
We assign weights to vertices and edges of $\G\bs \sT$ by 
\begin{equation}\label{eq_weightsGn}
w(\tilde{v})=[\G_v:Z(\F_q)\cap \G]
\quad \text{and}\quad
w(\tilde{e})=[\G_e:Z(\F_q)\cap \G],  
\end{equation}
where $v$ (resp. $e$) is a preimage of $\tilde{v}$ (resp. $\tilde{e}$).   
It is clear that this is well-defined, and $w(\tilde{e})$ divides both $w(t(\tilde{e}))$ and $w(o(\tilde{e}))$.  

Denote by $\cH(\sT, R)^\G$ the subgroup of $\G$-invariant harmonic cochains, i.e.,  
$f|\gamma=f$ for all $\gamma\in \G$.
It is clear that $f\in \cH(\sT, R)^\G$ defines a function $f'$ on the quotient graph $\G\bs\sT$, and 
$f$ itself can be uniquely recovered from this function: If $e\in E(\sT)$ maps to $\tilde{e}\in E(\G\bs \sT)$ under the quotient map, 
then $f(e)=f'(\tilde{e})$. The group of $R$-valued \textit{cuspidal 
harmonic cochains} for $\G$, denoted $\cH_0(\sT, R)^\G$, is the 
subgroup of $\cH(\sT, R)^\G$ consisting of functions 
which have compact support as functions on $\G\bs\sT$, i.e., functions 
which assume value $0$ on all but finitely many edges of $\G\bs\sT$.  
The orientation on $\sT$ does not necessarily descent to an orientation on $\G\bs\sT$, 
but we fix some orientation $E(\G\bs\sT)^+$ and define a pairing on $\cH_0(\sT, \Z)^\G$ by 
\begin{equation}\label{eqPIP}
(f, g)=\sum_{e\in E(\G\bs\sT)^+}f(e)g(e)w(e)^{-1}. 
\end{equation}
Since $f$ and $g$ are cuspidal, all but finitely many terms of this sum are zero, so the pairing is well-defined. 
It is clear that $(\cdot, \cdot)$ is symmetric and positive-definite. It is also $\Z$-valued, as follows from \cite[(5.7.1)]{GR}. 

We will primarily work with $\G=\G_0(\fn)$. To simplify the notation, we  put $$\cH_0(\fn, R):=\cH_0(\sT, R)^{\G_0(\fn)}.$$ 
It is known that $\cH_0(\fn, \Z)$ is a free $\Z$-module of rank equal to the genus of $X_0(\fn)$; cf. \cite[p. 49]{GR}.   
A $1$-cycle $\varphi\in H_1(\G_0(\fn)\bs \sT, \Z)$ can be thought of as a $\G_0(\fn)$-invariant function $\varphi: E(\sT)\to \Z$. 
Then $\varphi^\ast: e\mapsto w(e)\varphi(e)$ is in $\cH_0(\fn, \Z)$ and 
\begin{align}\label{thmGN}
j: H_1(\G_0(\fn)\bs\sT, \Z) &\xrightarrow{\sim} \cH_0(\fn, \Z), \\ 
\nonumber \varphi &\mapsto \varphi^\ast
\end{align}
is an isomorphism by \cite{GN}. The following is straightforward: 

\begin{lem}\label{lemImmediate}
For the weighted pairing \eqref{eqweighte} on $H_1(\G_0(\fn)\bs\sT, \Z)$ and \eqref{eqPIP} 
on $\cH_0(\fn, \Z)$ we have 
$(\varphi, \psi)=(\varphi^\ast, \psi^\ast)$. 
\end{lem}

\begin{rem}\label{remPIP}
The Haar measure on $\GL_2(\Fi)$ induces a push-forward measure on $E(\G\bs \sT)$, which, up to a scalar multiple, 
is equal to $w(e)^{-1}$; cf. \cite[(4.8)]{GR}. 
One can show that \eqref{eqPIP} agrees with the restriction to $\cH_0(\sT, \Z)^\G$ of the 
Petersson inner-product if one interprets $\cH_0(\sT, \C)^\G$ as a space 
of automorphic forms; see \cite[5.7]{GR}.
\end{rem}

\subsection{Hecke operators} Fix a non-zero ideal $\fn\lhd A$. Given a non-zero ideal $\fm\lhd A$, define 
an $R$-linear transformation of the space of $R$-valued functions on $E(\sT)$, the \textit{$\fm$-th Hecke operator}, by 
\begin{equation}\label{eqDefTm}
f|T_\fm=\sum f|\begin{pmatrix} a & b \\ 0 & d\end{pmatrix},
\end{equation}
where the sum is over $a,b,d\in A$ such that $a,d$ are monic, $(ad)=\fm$, $(a)+\fn=A$, and $\deg(b)< \deg(d)$. 
The Hecke operators preserve $\cH_0(\fn, R)$ and have the usual properties: They commute, 
satisfy $T_{\fm\cdot \fm'}=T_\fm\cdot T_{\fm'}$ for $\fm$ and $\fm'$ coprime, 
for a prime $\fp$, $T_{\fp^i}$ is a polynomial with integral coefficients in $T_\fp$, and 
$T_\fm$ is self-adjoint with respect to the pairing \eqref{eqPIP} if $\fm$ 
is coprime to $\fn$. Let $\T(\fn)$ be the commutative $\Z$-subalgebra of $\End_\Z(\cH_0(\fn, \Z))$ 
generated by all Hecke operators. 

The harmonic cuspidal cochains $\cH_0(\fn, \Z)$ have Fourier expansions, where 
the Fourier coefficients $c_\fm(f)$ of $f\in \cH_0(\fn, \Z)$ are indexed by the non-zero ideals $\fm\lhd A$; 
cf. \cite[pp. 42-43]{Analytical}. 
In \cite{Analytical}, Gekeler shows that $c_1(f)=-f\left(\begin{pmatrix} \varpi_\infty^2 & \varpi_\infty \\ 0 & 1\end{pmatrix}\right)$ and 
the bilinear pairing 
\begin{align}\label{GPairing}
\T(\fn)\times \cH_{0}(\fn, \Z) \to \Z\\
\nonumber t, f \mapsto c_1(f|t).
\end{align} 
is $\T(\fn)$-equivariant, non-degenerate, and becomes a perfect pairing after tensoring with $\Z[p^{-1}]$. 

\begin{rem}
It is not known if in general the pairing \eqref{GPairing} is perfect, without inverting $p$. 
This is in contrast to the situation over $\Q$ where the analogous pairing 
between the Hecke algebra and the space of weight-$2$ cusp forms on $\G_0(N)$ with 
integral Fourier expansions is perfect (cf. \cite[Thm. 2.2]{RibetModp}). 
In \cite{PW2}, it it shown that \eqref{GPairing} is perfect if $\deg(\fn)=3$. 
\end{rem}

Let $h=\rank_\Z  \cH_{0}(\fn, \Z)$. 
Because \eqref{GPairing} 
is non-degenerate, $\T(\fn)$ is a commutative $\Z$-algebra which as a $\Z$-module is free of rank $h$. 
Let $t_1, \dots, t_h$ be a $\Z$-basis of $\T(\fn)$. After fixing a $\Z$-basis of $\cH_{0}(\fn, \Z)$, every 
Hecke operator can be represented by a matrix. For $M\in \mathrm{Mat}_{h\times h}(\Z)$, 
let $\Tr(M)$ denote its trace. The \textit{discriminant of $\T(\fn)$} is 
\begin{equation}\label{deffDT}
\fD_{\T(\fn)}=|\det(\Tr(t_it_j))_{1\leq i, j\leq h}|. 
\end{equation}
The discriminant $\fD_{\T(\fn)}$ does not depend on the choice of a basis of $\T(\fn)$ or $\cH_0(\fn, \Z)$; see 
\cite[p. 49]{SerreLF} or \cite[p. 66]{Reiner}. 

Let $G(\fn)$ denote the graph $\G_0(\fn)\bs \sT$ with weights \eqref{eq_weightsGn}. This graph 
is not finite, but $H_1(G(\fn), \Z)$ has finite rank, so the discriminant $\fD_\fn:=\fD_{G(\fn), w}$ is defined. 
Let $\phi_1, \dots, \phi_h$ be a $\Z$-basis of $\cH_0(\fn, \Z)$. From Definition \ref{defnDisc}, \eqref{thmGN} 
and Lemma \ref{lemImmediate}, we get 
$$
\fD_{\fn}=|\det((\phi_i, \phi_j))_{1\leq i, j\leq h}|, 
$$
where $(\phi_i, \phi_j)$ is the Petersson inner-product \eqref{eqPIP}. 

\begin{defn}
We say that $f\in \cH_0(\fn, \R)$ is a normalized eigenform if $f$ is an 
eigenvector for all $t\in \T(\fn)$ and $c_1(f)=1$. 
\end{defn}

Assume $\fn=\fp$ is prime. The function field analogue of the theory of Atkin and Lehner \cite{AL} 
implies that $\cH_0(\fp, \R)$ has a basis consisting of normalized eigenforms. We extend the  
pairing \eqref{eqPIP} to $\cH_0(\fn, \R)$. 

\begin{thm}\label{thmfDT} 
Assume the pairing \eqref{GPairing} is perfect for $\fn=\fp$. Then 
$$
\fD_\fp\fD_{\T(\fp)}=\prod_{i=1}^h (f_i, f_i)^2,  
$$
where $\{f_1, \dots, f_h\}$ is a basis of $\cH_0(\fp, \R)$ consisting of normalized eigenforms.  
\end{thm} 
\begin{proof}
The argument that we present is very similar to the proof of Theorem 4.1 in \cite{Ullmo}. 

The map 
\begin{equation}\label{eqTpform}
\T(\fp) \otimes\R \to \R^h, \quad t \mapsto (a_1(f_1|t), \dots, a_1(f_h|t))
\end{equation}
is an isomorphism of $\R$-algebras. The trace form on $\T(\fp)$ corresponds to the 
standard scalar product on $\R^h$. Let $\Vol$ be the standard volume form on $\R^h$. 
Then $\Vol(\T(\fp))^2=\fD_{\T(\fp)}$, where by abuse of notation we denote by $\T(\fp)$ 
the image of the lattice $\T(\fp)\subset \T(\fp) \otimes\R$ under \eqref{eqTpform}.  

Now consider the isomorphism 
$\R^h\to \cH_0(\fp, \R)$ mapping the standard basis of $\R^h$ to $\{f_1, \dots, f_h\}$. 
It is known that the eigenforms $\{f_1, \dots, f_h\}$ are orthogonal to each other 
with respect to \eqref{eqPIP}, i.e., $(f_i, f_j)=0$ if $i\neq j$. 
Let $\Vol'$ denote the volume form on $\cH_0(\fp, \R)$ corresponding 
 to the scalar product \eqref{eqPIP}. Then 
 $$
 \Vol(\T(\fp))=\Vol'(\T(\fp)) \prod_{i=1}^h (f_i, f_i)
 $$
On the other hand, $\Vol'(\cH_0(\fp, \Z))^2=\fD_\fp$, and since \eqref{GPairing} is assumed to be perfect
$$
\Vol'(\T(\fp))\cdot \Vol'(\cH_0(\fp, \Z))=1. 
$$
Combining these volume calculations, we get the formula of the theorem. 
\end{proof}

\begin{thm}\label{thmPIPest} There are positive constants $c_1$ and $c_2$, depending only on $q$, such that 
for any normalized eigenform $f\in \cH_0(\fp, \R)$,
$$
c_1 \frac{|\fp|}{\deg(\fp)}\leq (f, f)\leq c_2 |\fp| (\deg(\fp))^3. 
$$
\end{thm}
\begin{proof} Using the Rankin-Selberg method, the Petersson norm $(f, f)$ can be related to 
a special value of the $L$-function of the symmetric square of $f$, which can be 
estimated using analytic methods. For the details we refer to 
Equation (18) and Proposition 5.5 in \cite{PapikianDC}, and Theorem 4.6 in \cite{PapikianDC2}. 
\end{proof}


\subsection{Jacobians of relative curves}\label{ssJRC} Let $C$ be a smooth, projective, geometrically connected curve of genus $g_C$
defined over $\F_q$. Let $F$ be the function field of $C$. 
Let $\pi: \cX\to C$ be a semi-stable curve of genus $g\geq 2$ over $C$. Recall that this means that $\pi$ 
is a proper and flat morphism whose fibres 
$\cX_{\bar{s}}$ over the geometric points $\bar{s}$ of $C$ are reduced, connected curves of arithmetic genus $g$, 
and have only ordinary double points as singularities; cf. \cite[p. 245]{NM}. We assume that 
the generic fibre $X:=\cX_F$, as a projective curve over $F$, is smooth and non-isotrivial.
Let $J:=\Pic^0_{X/F}$ be the Jacobian of $X$; cf. \cite[p. 243]{NM}. Let $\cJ\to C$ be the N\'eron model 
of $J$, and $\cJ^0$ be the connected component of the identity of $\cJ$. The assumption that $\cX\to C$ is semi-stable 
is equivalent to $(\cJ^0)_{\bar{s}}$ being a semi-abelian variety for all $\bar{s}$; see \cite[p. 246]{NM} and \cite[Prop. 5.7]{Deschamps}. 

Let $e_\cJ: C\to \cJ$ be the unit section of $\cJ\to C$, and $\Omega^1_{\cJ/C}$ be the sheaf of relative differential forms. 
The sheaf $e_\cJ^\ast(\Omega^1_{\cJ/C})$ on $C$ is locally free of rank $g$. The Parshin \textit{height} 
of $J$ is $$H(J):=\deg \bigwedge^g e_\cJ^\ast(\Omega^1_{\cJ/C}).$$ 

\begin{thm}\label{thmH1}
If $\pi: \cX\to C$ is the minimal 
regular model of $X$ over $C$, and $\omega_{\cX/C}$ is the relative dualizing sheaf, then 
$$
H(J)=\deg(\pi_\ast (\omega_{\cX/C}))=\frac{1}{12}\left(\omega_{\cX/C}\cdot \omega_{\cX/C}+\sum_{s\in C}\varrho_s\deg(s)\right), 
$$
where the sum is over the closed points of $C$ and $\varrho_s$ denotes the number of singular points in the fibre $\cX_s:=\pi^{-1}(s)$. 
\end{thm}
\begin{proof}
See \cite[p. 48]{Szpiro}. 
\end{proof}

\begin{thm}\label{thmH2}
Assume $\pi: \cX\to C$ is semi-stable and non-isotrivial. Let $\omega_{\cX/C}$ be the relative dualizing sheaf. 
Then 
$$
0\leq \omega_{\cX/C}\cdot \omega_{\cX/C}\leq 8p^e g(g-1)(g_C-1+\theta/2), 
$$
where $\theta$ is the number of geometric points of $C$ where the fibres of $\pi$ are not smooth, and 
$e$ is the modular inseparable exponent of $\pi$ as defined in \cite[p. 46]{Szpiro}. 
\end{thm}
\begin{proof}
See Proposition 1 and Theorem 3 in \cite{Szpiro}. 
\end{proof}

Let $s\in C$ be a closed point, and $x\in X_s$ be a singular point. There exists 
a scheme $S'$, \'etale  over $S:=\Spec(\cO_{C, s})$, such that any point $x'\in X':=X\times_S S'$ 
lying above $x$, belonging to a fiber $X_{s'}'$, is a split ordinary double point, and 
$$
\widehat{\cO}_{X', x'}\cong \widehat{\cO}_{S', s'}[\![u, v]\!]/(uv-c)
$$
for some $c\in \cO_{S', s'}$. Moreover, the valuation $w_x$ of $c$ for the normalized valuation of $\cO_{S', s'}$ 
is independent of the choice of $S', s'$, and of $x'$. For the proof of these facts we refer to \cite[Cor. 10.3.22]{Liu}. 

One can associate a graph $G_{X_s}$ to $X_s$, the so-called \textit{dual graph} (cf. \cite[p. 511]{Liu}): 
Let $k_s$ be the residue field at $s$. 
The vertices of $G_{X_s}$ are the irreducible components of $X_s\times_{k_s} \overline{k_s}$, and each 
ordinary double point $x\in X_s$ defines an edge $e_x$ whose end-vertices correspond to the irreducible components containing $x$ 
(the two orientations of $e_x$ correspond to a choice of one of the two branches passing through $x$ as the 
origin of $e_x$). We assign the weight $w(e_x)=w_x$.  

\begin{thm}\label{thmGroth}
Let $\Phi_{J, s}:=\cJ_s/\cJ_s^0$ be the group of connected components of $J$ at $s\in C$. Then $|\Phi_{J,s}|=\fD_{G_{X_s}, w}$. 
\end{thm}
\begin{proof}
This follows from 11.5 and 12.10 in \cite{SGA7}. 
\end{proof}

\begin{rem}\label{rem15}
Let $\widetilde{X}\to X$ be the minimal desingularization. The dual graph $G_{\widetilde{X}_s}$ 
is obtained from $G_{X_s}$ by replacing each $e_x\in E(G_{X_s})$ by a path without backtracking of length $w_x$
and assigning weight $1$ to all edges of the resulting graph; cf. \cite[Cor. 10.3.25]{Liu}. 
\end{rem}

\section{Eigenvalues of Laplacian}\label{secDiagrams}

The notation and assumptions in this section are those 
of $\S$\ref{sec1}. In particular, $G$ is a weighted connected graph. 
 
\subsection{Discriminant and eigenvalues} 
Let $V$ be a finite dimensional vector space over $\Q$. A \textit{lattice} of $V$ is a 
$\Z$-submodule $\La$ of $V$ that is finitely generated and spans $V$. 
Following \cite[$\S$III.1]{SerreLF}, for an arbitrary pair of lattices 
$\La_1$ and $\La_2$ in $V$ define a function $\chi(\La_1,\La_2)$ as follows: 
Pick a sublattice $\La\subset \La_1\cap \La_2$, and put  
$$
\chi(\La_1, \La_2):=\frac{|\La_1/\La|}{|\La_2/\La|}. 
$$ 
By \cite[Lem. 1, p. 47]{SerreLF}, $\chi(\La_1, \La_2)$ does not depend on the choice of $\La$. Moreover, by 
\cite[Prop. 1, p. 48]{SerreLF}, the following formula is valid:
$$
\chi(\La_1, \La_2)\cdot \chi(\La_2, \La_3)=\chi(\La_1, \La_3). 
$$
 
\begin{thm}\label{thmMF} Assume $G$ is finite with $n$ vertices. Let 
$$
0< \la_1\leq \la_2\leq \cdots \leq \la_{n-1}
$$ 
be the non-zero eigenvalues of $\Delta$. Then 
$$
\fD_{G,w} \sum_{v\in V(G)}\prod_{v'\neq v}w(v')=\prod_{i=1}^{n-1}\la_i \prod_{e\in E(G)^+}w(e). 
$$
\end{thm}
\begin{proof}
To simplify the notation, let $C_i:=C_i(G, \Z)$, $C_i^\vee:=\Hom(C_i, \Z)$, $H_1:=H_1(G, \Z)$, and $H_1^\vee:=\Hom(H_1, \Z)$. 
Let 
$$
\widetilde{C}_1:=\{y\in C_1(G, \Q)\ |\ (x, y)\in \Z \text{ for all }x\in C_1\}
$$
be the codifferent of $C_1$; this is a lattice in $C_1(G, \Q)$.  
Let $C_0':=\ker(\eps)$. Consider the diagram 
$$
\xymatrix{
& &  & \widetilde{C}_1 & & & \\
& 0 \ar[r] & H_1\ar[r] & C_1\ar@{^{(}->}[u]  \ar[r]^-{\partial}\ar[d]_-{\phi} & C_0 \ar[r]^-{\eps} & \Z \ar[r] & 0 \\
0\ar[r] & \Z\ar[r] & C_0^\vee \ar[r]^-{\partial^\vee} & C_1^\vee\ar@/_1pc/@{-->}[uu]_(.3){\phi^{-1}}  \ar[r] & H_1^\vee \ar[r] & 0\\
& & C_0 \ar[u]^-{\pi} &  &  & 
}
$$
where 
$$
\phi(e)= (e, \ast) \quad\text{and}\quad \pi(v)=\langle v, \ast\rangle. 
$$
In the diagram the horizontal lines are exact sequences, and 
$\phi^{-1}$ denotes the inverse of $\phi$ as an isomorphism $C_1\otimes \Q\to C_1^\vee\otimes \Q$.   
Note that $\phi^{-1}$ maps $C_1^\vee$ isomorphically onto $\widetilde{C}_1$, and 
$
\partial^\ast=\phi^{-1}\partial^\vee \pi$. 
Let $\partial^\ast(C_0)$ denote the image of $C_0$ under $\partial^\ast$, which we consider as a $\Z$-submodule of $\widetilde{C}_1$. 
It is easy to see that $H_1\cap \partial^\ast(C_0) = 0$, and $H_1\oplus \partial^\ast(C_0)$ is a lattice of $C_1\otimes \Q$. 

The formula in the theorem will follow by computing $\chi(\widetilde{C}_1, H_1\oplus \partial^\ast(C_0))$ in two different ways. 
On one hand, by applying $\phi$, we get 
$$
\chi(\widetilde{C}_1, H_1\oplus \partial^\ast(C_0))=\chi(C_1^\vee, \phi(H_1)\oplus \partial^\vee\pi(C_0))
$$
$$
=\chi(C_1^\vee, \phi(H_1)\oplus \partial^\vee(C_0^\vee))\cdot \chi(\phi(H_1)\oplus \partial^\vee(C_0^\vee), \phi(H_1)\oplus \partial^\vee\pi (C_0))
$$
$$
=\chi(H_1^\vee, \phi(H_1))\cdot \chi(\partial^\vee(C_0^\vee), \partial^\vee\pi (C_0)),   
$$
so Lemma \ref{lem4} implies 
\begin{equation}\label{eq2}
\chi(\widetilde{C}_1, H_1\oplus \partial^\ast(C_0))= \fD_{G,w}\cdot \chi(\partial^\vee(C_0^\vee), \partial^\vee\pi (C_0)). 
\end{equation}

On the other hand, 
$$
\chi(\widetilde{C}_1, H_1\oplus \partial^\ast(C_0))=\chi(\widetilde{C}_1, C_1)\cdot \chi(C_1, H_1\oplus \partial^\ast(C_0))
$$
$$
=\chi(\widetilde{C}_1, C_1)\cdot \chi(C_0', \Delta(C_0))
$$
$$
=\chi(\widetilde{C}_1, C_1)\cdot \chi(C_0', \Delta(C_0')) \cdot \chi(\Delta(C_0'), \Delta(C_0)). 
$$
Note that the restriction of $\Delta$ to $C_0'\otimes \Q$ is an invertible operator, so by Proposition 2 in $\S$III.1 of \cite{SerreLF}, 
$$
\chi(C_0', \Delta(C_0'))=\det(\Delta|_{C_0'\otimes \Q})=\prod_{i=1}^{n-1} \la_i. 
$$
It is clear that 
$$
\chi(\widetilde{C}_1, C_1) = |\widetilde{C}_1/C_1|=\prod_{e\in E(G)^+}w(e),
$$
since $\{e/w(e)\ |\ e\in E(G)^+\}$ is a basis of $\widetilde{C}_1$. Hence 
\begin{equation}\label{eq3}
\chi(\widetilde{C}_1, H_1\oplus \partial^\ast(C_0)) = \prod_{i=1}^{n-1} \la_i \left(\prod_{e\in E(G)^+}w(e)\right) \chi(\Delta(C_0'), \Delta(C_0)).
\end{equation}

It remains to compute 
$\chi(\Delta(C_0'), \Delta(C_0))$ and $\chi(\partial^\vee(C_0^\vee), \partial^\vee\pi (C_0))$. 
We have 
$$
\chi(\Delta(C_0'), \Delta(C_0))^{-1}=\chi(\Delta(C_0), \Delta(C_0'))=|\Delta(C_0)/\Delta(C_0')|. 
$$
Since $C_0=C_0'\oplus \Z v_0$ for a fixed vertex $v_0$, we see that $\Delta(C_0)/\Delta(C_0')\cong \Z/N\Z$ 
is cyclic generated by $\Delta v_0$. 
Let 
$$
f=\sum_{v\in V(G)} \prod_{v'\neq v}w(v') v.
$$ 
It is easy to check that $\Delta f=0$. Let $d=\underset{v\in V(G)}{\gcd}(\prod_{v'\neq v}w(v'))$. Then $f_0:=f/d$ 
is a primitive element in $C_0$ which generates $\ker{\Delta}$. Let 
$$
r=\frac{1}{d}\sum_{v\in V(G)} \prod_{v'\neq v}w(v'). 
$$
Since $r v_0-f_0\in C_0'$, we have $r\Delta(v_0)=r\Delta(v_0)-\Delta(f_0)=\Delta(rv_0-f_0)\in \Delta(C_0')$.
This implies that $N$ divides $r$. On the other hand, 
$N\Delta v_0\in \Delta(C_0')$ implies that there exists some $f'\in C_0'$ such that $\Delta(Nv_0)=\Delta(f')$. 
Thus, $Nv_0-f'\in \ker(\Delta)$. But the kernel of $\Delta$ in $C_0$ is generated by $f_0$. Hence $Nv_0-f'=sf_0$ for some $s\in \Z$. 
Applying $\eps$ to both sides, we get $N=sr$, so $r\mid N$. Combining this with $N\mid r$, we get $N=r$. Therefore, 
\begin{equation}\label{eq4}
\chi(\Delta(C_0'), \Delta(C_0)) = 1/r. 
\end{equation}

Let $m=|E(G)^+|$. Since $G$ is connected, $m\geq n-1$. 
By fixing an ordering of $V(G)$ and $|E(G)^+|$, we can think of $\partial^\vee(C_0^\vee)$ as the 
submodule of $\Z^m$ generated by the rows of an $n\times m$ matrix $M$ with entries in $\Z$, 
whose rows are labelled by the vertices and columns by the edges. Since $\ker(\partial^\vee)$ has $\Z$-rank $1$, 
the rank of $M$ is $n-1$. Note that $C_1^\vee/\partial^\vee(C_0)\cong H_1^\vee$ is a free $\Z$-module.  
Hence by a well-known fact from linear algebra (cf. \cite[p. 88]{Prasolov}) the greatest common divisor of minors of order $n-1$ of $M$ 
is equal to $1$. Let $D$ be the $n\times n$ diagonal matrix whose $(i,i)$th entry with respect to our ordering of vertices 
is $w(v_i)$. Now $\partial^\vee\pi (C_0)$ is the 
submodule of $\Z^m$ generated by the rows of $DM$. Hence $\chi(\partial^\vee(C_0^\vee), \partial^\vee\pi (C_0))$ 
is equal to the order of the torsion subgroup of $C_1^\vee/\partial^\vee\pi (C_0)$. 
In matrix terminology, this latter number is equal to the greatest common divisor of the minors of order $n-1$ of $DM$. 
It is easy to see that this greatest common divisor is equal to $d$ times the greatest common divisor of the minors of order $n-1$ of $M$. 
Thus, 
\begin{equation}\label{eq5}
\chi(\partial^\vee(C_0^\vee), \partial^\vee\pi (C_0)) = d. 
\end{equation}

Now the claim of the theorem easily follows by combining equations \eqref{eq2}-\eqref{eq5}. 
\end{proof}

\begin{cor}
If $w(v)=1$ for all $v\in V(G)$, then 
$$
\fD_{G,w}=\frac{1}{n}\prod_{i=1}^{n-1}\la_i \prod_{e\in E(G)^+}w(e). 
$$
\end{cor}

\begin{figure}
\begin{tikzpicture}[semithick, node distance=1.5cm, inner sep=.5mm, vertex/.style={circle, fill=black}]
\node[vertex] [label=left:$v_1$](1){};
  \node (2) [right of=1] {$\vdots$};
  \node[vertex] (3) [right of=2, label=right:$v_2$] {};
  \draw (1) edge[bend right]  (3) (1) edge[bend right=70] node [below] {$e_m$}  (3) (1) edge[bend left]  (3) (1) edge[bend left=70] node [above] {$e_1$} (3); 
\end{tikzpicture}
\caption{}\label{Fig1}
\end{figure}

\begin{example}\label{exBanana}
Let $G$ be a graph consisting of two vertices joined by $m$ edges; see Figure \ref{Fig1}.
Let $w_i=w(e_i)$, and assume $w(v_1)=w(v_2)=1$. Then $$\Delta(v_1-v_2)=2\left(\sum_{i=1}^m w_i^{-1}\right)(v_1-v_2).$$
Hence the non-zero eigenvalue is $2\sum_{i=1}^m w_i^{-1}$ and 
$$
\fD_{G,w}=\sum_{i=1}^m \prod_{j\neq i}w_j. 
$$
Combining this calculation with Theorem \ref{thmGroth} gives an alternative proof of \cite[Cor. 9.6/10]{NM}. 
\end{example}

\subsection{Diagrams} In this subsection we investigate the relationship between the spectra of 
certain infinite graphs and their finite subgraphs. 

\begin{defn}\label{defnDiagram}
Let $G$ be a weighted (possibly infinite) graph as in $\S$\ref{sec1}. 
We make the following assumptions: 
\begin{enumerate} 
\item[(i)] $G$ is bipartite, i.e., $V(G)$ is a disjoint union $V(G)=O\sqcup I$ such that 
any edge $e\in E(G)$ has one of its end-vertices in $O$ and the other in $I$; 
\item[(ii)] $w(e)$ divides $w(t(e))$ for any $e\in E(G)$ (hence also $w(e) \mid w(o(e))$); 
\item[(iii)] there is a positive integer $q$ such that for any $v\in V(G)$  
$$
\sum_{\substack{e\in E(G)\\ t(e)=v}} \frac{w(v)}{w(e)}=q+1; 
$$
\item[(iv)] $\sum_{v\in V(G)} w(v)^{-1}<\infty$. 
\end{enumerate}
In \cite{Morgenstern}, a graph with these properties is called a \textit{$(q+1)$-regular diagram}. 
\end{defn}

\begin{lem}\label{lemI=O}
$\sum_{v\in I}w(v)^{-1}=\sum_{v\in O}w(v)^{-1}$. 
\end{lem}
\begin{proof}
Since $G$ is bipartite and $w(e)=w(\bar{e})$, using property (iii) we get 
$$
\sum_{v\in I}\frac{q+1}{w(v)}=\sum_{v\in I} \sum_{\substack{e\in E(G)\\ t(e)=v}} \frac{1}{w(e)}
= \sum_{\substack{e\in E(G)\\ t(e)\in I}}\frac{1}{w(e)}= \sum_{\substack{e\in E(G)\\ t(e)\in O}}\frac{1}{w(e)}=\sum_{v\in O}\frac{q+1}{w(v)}. 
$$
\end{proof}

Let $L_2(G)$ be the Hilbert space of complex valued functions on $V(G)$ with inner product 
$$
\langle f, g\rangle=\sum_{v\in V(G)} f(v)\overline{g(v)}w(v)^{-1}.
$$
The \textit{adjacency operator} $\delta: L_2(G)\to L_2(G)$ is defined by 
\begin{equation}\label{eqAdjOp}
\delta(f)(v)=\sum_{\substack{e\in E(G)\\ t(e)=v}} \frac{w(v)}{w(e)}f(o(e)). 
\end{equation}
This operator is Hermitian, since by expanding we have 
$$
\langle \delta f, g\rangle = \sum_{e\in E(G)} f(o(e))\overline{g(t(e))} w(e)^{-1} = \langle f, \delta g\rangle. 
$$
By the Schur test \cite[p. 30]{Conway}, $\delta$ is bounded by 
\begin{equation}\label{eqNormdelta}
\norm{\delta}\leq q+1.
\end{equation} 
(It is clear that $\delta$ is not compact if $G$ is infinite.) 

If $f$ is a constant function, i.e., $f(v)= f(v')$ for all $v,v'\in V(G)$, 
then $\delta f=(q+1)f$. If $f$ is an alternating function, i.e., $f(v)=-f(v')$ for all $v\in I, v'\in O$, 
then $\delta(f)=-(q+1)f$. The orthogonal complement in $L_2(G)$ of the subspace spanned by 
the constant and alternating functions is 
\begin{equation}\label{eqL20}
L_2^0(G)=\left\{f\in L_2(G)\ \bigg | \sum_{v\in I}f(v)w(v)^{-1}=\sum_{v\in O}f(v)w(v)^{-1}=0 \right\}. 
\end{equation}

Let 
$$
m=\inf_{\substack{f\in L_2^0(G)\\ \norm{f}=1}} \langle \delta f, f\rangle,\quad \quad M=\sup_{\substack{f\in L_2^0(G)\\ \norm{f}=1}} \langle \delta f, f\rangle. 
$$
Since $\delta$ is Hermitian and bounded, the spectrum of $\delta$ on $L_2^0(G)$ lies in the closed 
interval $[m, M]$ on the real axes; cf. Theorem 9.2-1 in \cite{Kreyszig}. Moreover, $m$ and $M$ are spectral values of $\delta$, and 
$$\norm{\delta|_{L_2^0(G)}}=\max(|m|, |M|);$$ cf. Theorems 9.2-2 and 9.2-3 in \cite{Kreyszig}. From \eqref{eqNormdelta}, we clearly have 
$$
\max(|m|, |M|) \leq q+1. 
$$
\begin{lem}\label{lemSym}
The spectrum of $\delta$ on $L_2^0(G)$ is symmetric with respect to zero. In particular, $m=-M$. 
\end{lem}
\begin{proof} Let $\la$ be in the spectrum of $\delta|_{L_2^0(G)}$. By definition, this means that the 
linear operator $\delta-\la I$ is not bijective; cf. \cite[pp. 371-373]{Kreyszig}. This can happen in two ways, 
either $\delta-\la I$ is not injective, or $\delta-\la I$ is not surjective. 

First, assume $\delta-\la I$ is not injective. Then $\la$ is an eigenvalue of $\delta$. 
Let $\delta f=\la f$ be an eigenfunction. We 
write $f=f_0+f_1$, where $f_0$ is supported on $O$ and $f_1$ is supported on $I$; note that 
such decomposition is unique.  
Since $\delta f_0$ (resp. $\delta f_1$) is supported on $I$ (resp. $O$), we must have $\delta f_0=\la f_1$ 
and $\delta f_1=\la f_0$. Now 
$$
\delta (f_0-f_1)=\la f_1 - \la f_0 = -\la (f_0-f_1). 
$$
It is clear from \eqref{eqL20} that if $f\in L_2^0(G)$, then $f_0-f_1$ is also in this subspace. 
Thus, $-\la$ is an eigenvalue of the restriction of $\delta$ to $L_2^0(G)$. 

Next, assume $\delta-\la I$ is not surjective on $L_2^0(G)$. Suppose $g$ is not in the image of $\delta-\la I$. 
Write $g=g_0+g_1$ as earlier. We claim that $g_1-g_0\in L_2^0(G)$ is not in the image of $\delta+\la I$, 
and so $-\la$ is also in the spectrum of $\delta$. Assume the contrary: there exists $h=h_0+h_1$ 
such that $\delta h +\la h=g_1-g_0$. Then 
$$
\delta h_0 +\la h_1=g_1 \quad \text{and} \quad \delta h_1+\la h_0=-g_0. 
$$
This can be rewritten as 
$$
\delta h_0 -\la (-h_1)=g_1 \quad \text{and} \quad \delta (-h_1)-\la h_0=g_0. 
$$
This implies $(\delta-\la I)(h_0-h_1)=g$, which is a contradiction. 
\end{proof}

Let $G'$ be a finite connected subgraph of $G$ with the property that if $v, v'\in V(G)$ are in $V(G')$ 
then any edge of $G$ connecting $v$ and $v'$ is also an edge of $G'$. The weights 
of vertices and edges of $G'$ are the same as in $G$. Let 
$$
\delta'(f)(v)=\sum_{\substack{e\in E(G')\\ t(e)=v}} \frac{w(v)}{w(e)}f(o(e))
$$
be the adjacency operator of $G'$. Any function $f$ on $V(G')$ can be extended to a function $\tilde{f}\in L_2(G)$ by setting 
$$
\tilde{f}(v)=
\begin{cases}
f(v) & \text{if }v\in V(G')\\
0 & \text{if } v\not \in V(G'). 
\end{cases}
$$
Define an inner product on the $\C$-vector space of functions on $V(G')$ by $\langle f, g\rangle:=\langle \tilde{f}, \tilde{g}\rangle$. 
We denote this inner product space by $L_2(G')$. It is easy to see that 
$$
\langle \delta' f, g\rangle = \langle \delta \tilde{f}, \tilde{g}\rangle = \langle \tilde{f}, \delta \tilde{g}\rangle = 
\langle f, \delta' g\rangle.
$$
Hence the linear operator $\delta'$ on $L_2(G')$ is Hermitian. This implies that the eigenvalues 
$$
\la_1\leq \la_2\leq \cdots\leq \la_n
$$
of $\delta'$ are real; here $n=|V(G')|$. 

\begin{thm}\label{thmWeylHilb} We have 
$$
-(q+1)\leq \la_1, \quad m\leq \la_2, \quad \la_{n-1}\leq M, \quad \la_n\leq q+1. 
$$
\end{thm}
\begin{proof}
Since $G'$ is bipartite, the argument in the proof of Lemma \ref{lemSym} shows that 
the spectrum of $\delta'$ is symmetric with respect to zero. In particular, $\la_1=-\la_n$ and $\la_2=-\la_{n-1}$. 
Since we also have $m=-M$, the first two inequalities imply the other two. 

Let $\delta' f=\la_1 f$ with $\norm{f}=1$. Then $\norm{\tilde{f}}=1$ and 
$$
-(q+1)= \inf_{\substack{x\in L_2(G)\\ \norm{x}=1}} \langle \delta x, x\rangle \leq \langle \delta \tilde{f}, \tilde{f}\rangle 
= \langle \delta' f, f\rangle =\la_1. 
$$

Let $H=L_2^0(G)\oplus \C\mathbf{1}$ be the orthogonal complement of alternating functions; the second 
factor in $H$ is spanned by the constant functions. (The orthogonality of constant and alternating functions follows from 
Lemma \ref{lemI=O}.)
We claim that for any $0\neq h\in H$, 
we have $\langle \delta h, h\rangle/\langle h, h\rangle\geq m$.  
Indeed, write $h=h_1+h_2\in H$, where $h_1\in L_2^0(G)$ and $h_2\in  \C\mathbf{1}$. 
If $h_1=0$, then $\langle \delta h, h\rangle/\langle h, h\rangle = (q+1)> m$. If $h_1\neq 0$, then we have 
$$
\frac{\langle \delta h, h\rangle}{\langle h, h\rangle} = \frac{\langle \delta h_1, h_1\rangle +(q+1)\langle h_2, h_2\rangle}
{\langle h_1, h_1\rangle +\langle h_2, h_2\rangle}\geq \frac{\langle \delta h_1, h_1\rangle }
{\langle h_1, h_1\rangle }\geq m,
$$ 
where the first inequality follows from the fact that $\langle \delta h_1, h_1\rangle/ \langle h_1, h_1\rangle< (q+1)$.

Now let $\delta' g=\la_2 g$ with $\norm{g}=1$. Let $H'$ be the subspace of $L_2(G)$ 
spanned by $\tilde{f}$ and $\tilde{g}$. We claim that for any $0\neq x\in H'$, we have 
$\langle \delta x, x\rangle/ \langle x, x\rangle \leq \la_2$.   
Write $x=y+z$ where $y=a\tilde{f}$ and $z=b\tilde{g}$. Since $\delta'$ is Hermitian, we have $y\perp z$, and 
$$
\frac{\langle \delta x, x\rangle}{\langle x, x\rangle} = \frac{\la_1\langle y, y\rangle+ \la_2\langle z, z\rangle}{\langle y, y\rangle+\langle z, z\rangle}  
\leq \la_2. 
$$

Consider the orthogonal projection $P:L_2(G)\to \C$ onto the $1$-dimensional space spanned by the alternating functions. 
The null-space of $P$ is $H$. On the other hand, since $H'$ is $2$-dimensional, there is a non-zero vector $x\in H\cap H'$. 
From the previous two paragraphs, we get 
$m\leq \langle \delta x, x\rangle/\langle x, x\rangle \leq \la_2$, as was required to show. 
\end{proof}

\begin{thm}[Weyl's inequalities]\label{thmWeylI} 
Let $A$ and $B$ be $n\times n$ Hermitian matrices, and $C=A+B$. Let the eigenvalues of $A, B, C$ form increasing sequences: 
$$
\alpha_1\leq \cdots \leq \alpha_n, \quad  \beta_1\leq \cdots \leq \beta_n, \quad  \gamma_1\leq \cdots \leq \gamma_n. 
$$
Then 
\begin{enumerate}
\item[(i)] $\gamma_i\geq \alpha_j+\beta_{i-j+1}$ for $i\geq j$; 
\item[(ii)] $\gamma_i\leq \alpha_j+\beta_{i-j+n}$ for $i\leq j$. 
\end{enumerate}
\end{thm}
\begin{proof}
See Theorem 34.2.1 in \cite{Prasolov}. 
\end{proof}

\begin{defn}\label{defnDegVer}
A vertex $v\in V(G')$ is a \textit{boundary vertex} if not all vertices adjacent to $v$ in $G$ are in $G'$. 
The \textit{degree} $v\in V(G')$ is 
$$
\deg_{G'}(v)=\sum_{\substack{e\in E(G')\\ t(e)=v}} \frac{w(v)}{w(e)}. 
$$
The degree of any $v\in V(G')$ is non-zero since $G'$ is connected. 
If $v$ is not a boundary vertex, then by an earlier assumption $\deg_{G'}(v)=q+1$. 
\end{defn}

Enumerate the vertices $\{v_1, \dots, v_n\}$ of $G'$, and consider the set of these vertices as a basis for 
$C_0(G', \C)$. Denote 
$$
d_i=\deg_{G'}(v_i), \quad 1\leq i\leq n.  
$$
We assume that the enumeration is done so that $d_1\leq d_2\leq \cdots\leq d_n$. Let $\Delta$ 
be the Laplacian of $G'$ defined in Definition \ref{def1}. Let 
$$
0=\gamma_1\leq \gamma_2\leq \cdots \leq \gamma_n
$$
be the eigenvalues of $\Delta$. 

\begin{thm}\label{thmMTFA} We have 
$$
\gamma_2 \geq d_1-M,\qquad 
\gamma_{n-1}\leq (q+1)+M, \qquad \gamma_n \leq 2(q+1). 
$$
\end{thm}
\begin{proof} We have $\Delta = D - \delta'$, 
where $D$ is the diagonal matrix $\mathrm{diag}(d_i)_{1\leq i\leq n}$.  
The operators $\delta'$ and $D$ are Hermitian on $L_2(G')$, so Weyl's inequalities,  
combined with the bounds of Theorem \ref{thmWeylHilb}, yield 
\begin{align*}
\gamma_2 &\geq d_1-\la_{n-1}\geq d_1-M,\\ 
\gamma_{n-1} &\leq d_n-\la_{2}\leq d_n-m\leq (q+1)+M, \\
\gamma_n &\leq d_n-\la_{1}\leq d_n+(q+1)\leq 2(q+1). 
\end{align*}
\end{proof}


\subsection{Ramanujan diagrams}\label{ssRD} 
We say that $G$ is a \textit{Ramanujan diagram}, if it is a $(q+1)$-regular diagram in the sense of Definition \ref{defnDiagram}, 
and the following extra conditions hold:
\begin{enumerate}
\item[(v)] $G$ is a union of a finite connected graph $G'$ and a finite number of half-lines $C_1, \dots, C_s$, called \textit{cusps}, so that 
$$E(G)=E(G')\cup E(C_1)\cup \cdots \cup E(C_s).$$ 
\item[(vi)] $C_j\cap C_k=\emptyset$ for any $j\neq k$.  
\item[(vii)] If $\{v_n^j\}_{n\geq 0}$ are the vertices of $C_j$ ($1\leq j\leq s$), then $V(G')\cap V(C_j)=\{v_0^j\}$. 
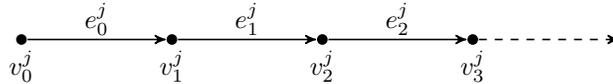
\begin{figure} [h]
\begin{tikzpicture}[->, >=stealth', semithick, node distance=2cm, inner sep=.5mm, vertex/.style={circle, fill=black}]
\node[vertex] (0) [label=below:$v_0^j$]{};
  \node[vertex] (1) [right of=0, label=below:$v_1^j$] {}; 
  \node[vertex] (2) [right of=1, label=below:$v_2^j$] {};
  \node[vertex] (3) [right of=2, label=below:$v_3^j$] {};
  \node[] (4) [right of=3] {};
\path[] (0) edge node[above]{$e_0^j$}  (1) (1) edge node[above]{$e_1^j$} (2) (2) edge node[above]{$e_2^j$} (3) (3) edge[dashed] (4);   
\end{tikzpicture}
\end{figure}
\item[(viii)] For $1\leq j\leq s$ and $n\geq 0$, let $e_n^j$ be the edge with origin $v_n^j$ and terminus $v_{n+1}^j$. Then 
\begin{align*}
w(v_n^j)/w(e_n^j) &=1,\\  w(v_{n+1}^j)/w(e_n^j) &=q.  
\end{align*}
\item[(ix)] $M\leq 2\sqrt{q}$. 
\end{enumerate} 

\begin{lem}\label{lemCuspform}
Let $G$ be a Ramanujan diagram as above. If $f\in L_2^0(G)$ is an eigenfunction 
for $\delta$, then $f$ vanishes on all $C_j$, i.e., $f(v_n^j)=0$ for $1\leq j\leq s$ and $n\geq 0$. 
This implies that $f$ is an eigenfunction for $\delta'$ with the same eigenvalue, and  
the discrete spectrum of $\delta$ is contained in the spectrum of $\delta'$. 
\end{lem}
\begin{proof} This observation appears in \cite[p. 178]{EfratInvent} (see also \cite[$\S$3]{Efrat}). 
Let $f\in L_2^0(G)$. Fix some cusp, and, to simplify the notation, denote its vertices by $v_n$, and $f(n):=f(v_n)$.  
The condition (viii) implies that $w(v_{n+1})/w(v_n)=q$, so for $f$ to be in $L_2(G)$ we must have 
$f(n)=o(q^{n/2})$ as $n\to \infty$. Assume $\delta f=\la f$. Then for $n\geq 1$ 
$$
\la f(n)=(\delta f)(n) = q f(n-1)+f(n+1). 
$$
Let $x_1, x_2$ be the roots of $x^2-\la x+q$. The above linear recurrence can be solved as $f(n)=a x_1^n+bx_2^n$ if $x_1\neq x_2$, or  
$f(n)=(a+nb) x_1^n$ if $x_1=x_2$ (here $a$ and $b$ are determined by $f(0)$ and $f(1)$). 
By condition (ix), the eigenvalue $\la$ satisfies $|\la|\leq 2\sqrt{q}$, so the roots are either $x_1=x_2=\pm\sqrt{q}$, 
or $x_1=\overline{x_2}$ are complex conjugate of absolute value $\sqrt{q}$. Unless $f(n)\equiv 0$, 
this implies that $f(n)/q^{n/2}$ does not tend to $0$ as $n\to \infty$, a contradiction. 
\end{proof}

\begin{figure} 
\begin{tikzpicture}[>=stealth', scale=2, semithick, inner sep=.5mm, vertex/.style={circle, fill=black}]
\node[vertex] (01) [label=left:$a$] at (0, .5) {};
\node[vertex] (10) [label=below:$b$] at (1, 0) {};
\node[vertex] (11) [label=above:$v_1'$] at (1, .5) {};
\node[vertex] (21) [label=above:$v_2'$] at (2, .5) {};
\node[vertex] (02) [label=above:$v_{\infty, 0}$] at (0, 1) {};
\node[vertex] (12) [label=above:$v_{\infty, 1}$] at (1, 1) {};
\node[vertex] (22) [label=above:$v_{\infty, 2}$]  at (2, 1) {};

\path[]
(01) edge[dashed]  (10)
(01) edge  (11)
(11) edge  (21)
(02) edge  (11)
(02) edge  (12)
(10) edge  (21)
(12) edge  (22);
\draw [->] (2,.5) -- (3,.5);
\draw [->] (2,1) -- (3,1);
\end{tikzpicture}
\caption{}\label{Fig2}
\end{figure}

\begin{example}\label{exampled3} Consider the diagram in Figure \ref{Fig2}. 
The dashed edge between the vertices $a$ and $b$ indicates that they are connected by 
$q$ edges, and the arrows indicate the cusps. 
The weights of $a$ and $b$ are $1$, so all edges having $a$ or $b$ as an end-vertex have weight $1$. 
The weights of $v_1'$ and $v_2'$ are $q-1$; the edge connecting $v_1'$ and $v_2'$ also has weight $q-1$.  
As we will explain later, $G$ is Ramanujan (see Remark \ref{remDeg3}). 

The graph $G'$ is the graph formed by the vertices $v_1', v_2', a, b$. The characteristic polynomial of $\delta'$ 
is 
$$
x^4-((q+1)^2-2)x^2+1. 
$$
Two of its roots have absolute value $< 2\sqrt{q}$, and the other two have absolute value lying in the interval $(2\sqrt{q}, q+1)$. 
Hence the spectrum of $G'$ is not in the spectrum of $G$. Moreover, it is easy to see that a function which vanishes on the cusps 
and is an eigenfunction of $\delta$ must be identically $0$. Thus, the discrete spectrum of $G$ is empty. 
\end{example}

\begin{rem}
 Given a diagram $G$, an eigenfunction $f\in L_2(G)$ of $\delta$ with finite support,  
i.e., $f(v)= 0$ for all but finitely many $v\in V(G)$, is called a \textit{cusp form}; cf. \cite[p. 177]{EfratInvent}. 
The point of Lemma \ref{lemCuspform} is that in case of a Ramanujan diagram $G$ the only 
eigenfunctions of $\delta$ in $L_2^0(G)$ are the cusp forms. In \cite{EfratInvent}, Efrat 
constructs examples of infinite diagrams which satisfy properties (i)-(viii), have no non-trivial cusp forms, 
but have lots of $\delta$-eigenfunctions in $L_2^0(G)$.    
\end{rem}


\section{Drinfeld diagrams}\label{sDD} 

\subsection{Ramanujan property} Let $\sT$ be the Bruhat-Tits tree of $\PGL_2(\Fi)$ 
as in $\S$\ref{ssHCHO}. Let $\G:=\GL_2(A)$. Let 
$\G'$ be a congruence subgroup of $\G$. We consider 
$\G'\bs \sT$ as a weighted infinite graph, with weights defined by \eqref{eq_weightsGn}. 

\begin{thm}\label{thmRDD}
The quotient graph $\G'\bs\sT$ is a $(q+1)$-regular Ramanujan diagram. 
\end{thm}
\begin{proof}
For $i\in \Z$, let $v_i\in V(\sT)$ be the vertex
represented by the matrix $\begin{pmatrix} \varpi_\infty^{-i} & 0 \\ 0 &1\end{pmatrix}$; it is easy to see 
that $v_i$ is adjacent to $v_{i+1}$.  
Let $e_i$ be the edge with $o(e_i)=v_i$, $t(e_i)=v_{i+1}$.
The subgraph formed by the $v_i$ and $e_i$ with $i\geq 0$ maps isomorphically 
onto the quotient graph $\G\bs \sT$; cf. \cite[p. 111]{SerreT}. 
Each orbit of the action of $\G$ on $V(\sT)$ splits into a disjoint union of orbits of $\G'$, and similarly 
for $E(\sT)$. This gives a natural covering 
$$
\pi: \G'\bs \sT\to \G\bs\sT. 
$$
Since $\G\bs\sT$ is bipartite, so is $\G'\bs\sT$, with the partition of vertices of $\G'\bs\sT$ induced by $\pi^{-1}$. 
Since $\G'_{\tilde{e}}$ is a subgroup of $\G'_{t(\tilde{e})}$, we see that $w(e)$ divides $w(t(e))$. 
Let $v$ be a fixed vertex of $\G'\bs \sT$, and $\tilde{v}$ be some vertex in $\sT$ mapping to $v$. 
The group $\G'_{\tilde{v}}$ acts on the set $\{\tilde{e}\ |\ t(\tilde{e})=\tilde{v}\}$, 
which has cardinality $(q+1)$. The orbit of a given edge $\tilde{e}$ under the action of $\G'_{\tilde{v}}$ 
has length $w(v)/w(e)$, where $e$ is the image of $\tilde{e}$ in $\G'\bs\sT$. This implies 
$\sum_{t(e)=v} w(v)/w(e)=q+1$. Next, according to \cite[p. 110]{SerreT} 
\begin{equation}\label{eqVol}
\sum_{v\in V(\G'\bs \sT)}w(v)^{-1}=\frac{2[\G:\G']\cdot |Z(\Fi)\cap \G'|}{(q^2-1)(q-1)^2}. 
\end{equation}
Note that $Z(\Fi)\cap \G'$ is a subgroup of $Z(\F_q)\cong \F_q^\times$. 
In particular, the series on the left converges. Overall, what we proved so far implies 
that $\G'\bs \sT$ is a $(q+1)$-regular diagram. 

We say that $e\in E(\G'\bs \sT)$ (resp. $v\in V(\G'\bs \sT)$) is of \textit{type $i$} if 
$\pi(e)=e_i$ (resp. $\pi(v)=v_i$). Denote  
\begin{align*}
V_i &=\{v\in V(\G'\bs\sT)\ |\ \mathrm{type}(v)=i\}, \\
E_i &=\{e\in V(\G'\bs\sT)\ |\ \mathrm{type}(e)=i\}. 
\end{align*}
Let 
\begin{align*}
G_0 &:=\GL_2(\F_q)\hookrightarrow \G \\
G_i &:=\left\{\begin{pmatrix} a & b \\ 0 & d\end{pmatrix}\in \G\ \big| \deg b\leq i \right\}\quad (i\geq 1). 
\end{align*}
For $i\geq 0$, $G_i$ is the stabilizer of $v_i$ in $\G$, and $G_i\cap G_{i+1}$ is the stabilizer of $e_i$; cf. \cite{GN}.  
The groups $G_i$ act on the set of cosets $\G/\G'$ from the left, and the orbits of various $G_i$ or $G_i\cap G_{i+1}$ 
correspond to the vertices or edges of $\G'\bs\sT$ of type $i$: 
\begin{align*} 
G_i\bs \G/\G' &\cong V_i \\ 
(G_i\cap G_{i+1})\bs \G/\G' &\cong E_i. 
\end{align*}

It is easy to see that (cf. \cite[p. 692]{GN})
\begin{equation}
 o: E_i\to V_i \text{ is bijective for }i\geq 1, 
 \end{equation}
 and, because $G_i\cap G_{i+1}=G_i$ for $i\geq 1$,
 \begin{equation}\label{eq33}
 w(e)=w(o(e)) \text{ for }e\in E_i, i\geq 1.
\end{equation}

Let $\fn\lhd A$ be of minimal degree $d$ such that $\G(\fn)$ is contained in $\G'$. Since 
$G_i$ acts on $\G/\G'$ via $p_\fn: G_i\to \G/\G(\fn)$ and $p_\fn(G_{d-1})=p_\fn(G_d)=\dots$, 
the subgraph of $\G'\bs \sT$ consisting of edges of type $\geq d-1$ is a disjoint union of half-lines. 
Since $|G_{i+1}/G_i|=q$ for $i\geq 1$, it is also clear that $w(t(e))=q\cdot w(o(e))$ for edges of type $\geq d-1$.  

To prove that $\G'\bs\sT$ is Ramanujan it remains to show that this graph has property (ix) 
in the definition of Ramanujan diagram. This is a rather deep fact, closely related to the theory of 
Eisenstein series and the Ramanujan--Petersson conjecture for automorphic representations of $\GL(2)$ over function fields   
proved by Drinfeld. The details can be found in \cite[Thm. 2.1]{Morgenstern}. (Although in \cite{Morgenstern}
it is assumed that $\G'=\G(\fn)$ is the principal congruence subgroup, the proof 
works also for other congruence subgroups.)
\end{proof}

\begin{rem}
\begin{enumerate} 
\item The sum \eqref{eqVol} can be interpreted as the volume of $\GL_2(\Fi)/\G'$ with respect to an  
appropriately normalized Haar measure on $\GL_2(\Fi)$; cf. \cite[p. 110]{SerreT}. 

\item The automorphic representations that arise at the end of the proof of Theorem \ref{thmRDD} 
are spherical at $\infty$, so these are not the automorphic representations that arise from Drinfeld modular curves 
which are special at $\infty$. 
\end{enumerate}
\end{rem}


\subsection{Number of vertices}\label{ss42.NV} Let $\fn\lhd A$ be a non-zero ideal of degree $d$. 
By Theorem \ref{thmRDD}, $\G_0(\fn)\bs \sT$ is a Ramanujan diagram. In particular, 
$\G_0(\fn)\bs \sT$ is a union of a finite graph $\cG$, and a finite number of cusps. 
As follows from the proof of Theorem \ref{thmRDD}, 
one can take $\cG$ to be the subgraph formed by vertices of type $\leq d-1$. On the 
other hand, this is not the most natural choice of $\cG$. Assume $\deg(\fn)\geq 3$, so that 
$H_1(\G_0(\fn)\bs \sT, \Z)\neq 0$. We choose $\cG$ to be the smallest subgraph of $\G_0(\fn)\bs \sT$ 
such that each cusp is attached to $\cG$ at a vertex which is adjacent in $\G_0(\fn)\bs \sT$ to at least $3$ vertices. 
It is easy to see that this finite graph $\cG$ is uniquely determined, and we denote it by $\cG_0(\fn)$. 

We want to apply Theorem \ref{thmMF} to $\cG_0(\fn)$. To do this, we need to compute the 
number of vertices and edges in $\cG_0(\fn)$, along with their weights. A large portion of 
this calculation is already contained in \cite{GN}, where one finds the number of vertices and edges of 
type $0$ and $d-1$.  

It is easy to see that 
\begin{align*}
\G/\G_0(\fn) & \overset{\sim}{\To} \p:=\p^1(A/\fn) \\ 
\begin{pmatrix} a & b \\ c & d\end{pmatrix} &\mapsto (a:c)
\end{align*}
as $\G$-sets, where the action of $\G$ on $\p$ is $\begin{pmatrix} a & b \\ c & d\end{pmatrix} (u:v)=(au+bv:cu+dv)$. 
Computing the number of vertices of type $i$ and their weights amounts to computing 
the orbits and stabilizers of the action of $G_i$ on $\p$. Similarly, computing 
the number of edges of type $i$ and their weights amounts to 
computing the orbits and stabilizers of the action of $G_i\cap G_{i+1}$ on $\p$. 
Since $G_i\cap G_{i+1}=G_i$ for $i\geq 1$, these two problems are the same for type $\geq 1$. 
For type $0$ one needs to consider the action on $\p$ of both $G_0$ and the group of upper-triangular matrices 
$B=G_0\cap G_1$ in $\GL_2(\F_q)$. 
The formulas become more and more complicated as the number of divisors of $\fn$ increases, so,  
for simplicity, from now on we assume $\fn=\fp$ is prime. Let 
$$
\kappa(\fp)=\begin{cases}
1 & \text{if $\deg(\fp)$ is even};\\
0 & \text{otherwise}. 
\end{cases}
$$
\begin{lem} There is one vertex $v_{\infty, 0}\in V_0$ of weight $q(q-1)$.  
There is one vertex $v'_0\in V_0$ of weight $q+1$ if $\kappa(\fp)=1$. The number of 
remaining vertices of type $0$ is 
$$
|V_0|-1-\kappa(\fp)=\frac{(q^{d-1}-1)-\kappa(\fp)\cdot (q-1)}{q^2-1},
$$
and they all have weight $1$. 
\end{lem}
\begin{proof}
This follows from Lemma 2.7 in \cite{GN}. 
\end{proof}

\begin{lem}\label{lem3.4}
There are two edges $e_\infty, e_\infty'\in E_0$ with origin $v_{\infty, 0}$. Their weights are $q(q-1)$ and $q-1$, respectively. 
When $\kappa(\fp)=1$, there is a unique edge with origin $v'_0$, and its weight is $1$. Any other vertex in 
$V_0$ is the origin of exactly $q+1$ edges, all of weight $1$.  
\end{lem}
\begin{proof}
This follows from Lemma 2.8 in \cite{GN}. 
\end{proof}

\begin{lem}\label{lem3.5}
Assume $1\leq i\leq d-1$. There is one vertex $v_{\infty, i}\in V_i$ of weight $(q-1)q^{i+1}$, and one 
vertex $v_i'$ of weight $q-1$. The number of remaining vertices of type $i$ is
$$
|V_i|-2=\frac{q^{d-1-i}-1}{q-1},
$$
and they all have weight $1$. 
\end{lem}
\begin{proof}
Let $\gamma=\begin{pmatrix} a & b \\ 0 & d\end{pmatrix}\in G_i$. Then $\gamma (1:0)=(1:0)$, hence 
$(1:0)$ is fixed by $G_i$. This gives the vertex $v_{\infty, i}$ with weight $|G_i|/(q-1)=(q-1)q^{i+1}$. 
Next, suppose 
$$
\gamma (u:1)=(\frac{ua+b}{d}:1)=(u:1). 
$$
Then $b=(d-a)u$. Note that $u$ is the residue class of a unique polynomial 
of degree $\leq d-1$. 

If $\deg(u)>i$, then $a=d$ and $b=0$ (as $\deg(b)\leq i$). In that case,  
the stabilizer of $(u:1)$ in $G_i$ is $Z(\F_q)\cong \F_q^\times$, and the orbit of $(u:1)$ has length $|G_i|/(q-1)$. 
Note that all elements in $G_i(u:1)$ are of the form $(u':1)$ with $\deg(u')=\deg(u)$. 
Hence the $q^{d}-q^{i+1}$ points $(u:1)\in \p$ with $\deg(u)>i$ give $(q^{d}-q^{i+1})/q^{i+1}(q-1)$ vertices 
of type $i$ and weight $1$. 

If $\deg(u)\leq i$, then $a,d\in \F_q^\times$ can be arbitrary, 
so the stabilizer of $(u:1)$ in $G_i$ is isomorphic to $\F_q^\times\times \F_q^\times$. The length 
of the orbits of $(u:1)$ is $|G_i|/(q-1)^2=q^{i+1}$. But there are exactly $q^{i+1}$ points $(u:1)\in \p$ 
with $\deg(u)\leq i$, so they are all in one orbit. This gives one vertex of type $i$ and weight $q-1$.  
\end{proof}

As follows from the previous proof, the vertices $v_{\infty, i}$ ($i\geq 0$) all come from $(1:0)\in \p$, 
so $v_{\infty, i}$ is adjacent to $v_{\infty, i+1}$. Moreover, it is easy to see that each $v_{\infty, i}$ 
is adjacent to exactly two vertices in $\G_0(\fp)\bs \sT$, so $\{v_{\infty, i}\}_{i\geq 0}$ form a cusp. The weight 
of any other vertex $v\in V_i(\G_0(\fp)\bs \sT)$, $1\leq i\leq d-1$, is $1$ or $q-1$. This implies that 
$v$ is adjacent to at least $3$ other vertices.  
\begin{defn}\label{defnG(p)}
Let $\cG_0(\fp)$ be the subgraph of $\G_0(\fp)\bs \sT$ 
formed by all vertices of type $\leq d-1$, excluding $v_{\infty, 0}, v_{\infty, 1}, \dots, v_{\infty, d-1}$. 
Note that $e_\infty'$ connects $v_{\infty, 0}$ to $\cG_0(\fp)$. (It is easy to see from previous  
discussions that $\cG_0(\fp)$ is the graph from the first paragraph of $\S$\ref{ss42.NV}.) 
\end{defn}

One easily computes from previous lemmas that 
\begin{align}
|V(\cG_0(\fp))| &=\frac{2q(q^{d-1}-1)}{(q-1)^2(q+1)}+\frac{(q-2)(d-1)}{q-1}+\frac{\kappa(\fp)q}{q+1} \\ \nonumber &=c_1 q^d +c_2 d+c_3, 
\end{align}
\begin{align}\label{align4.5}
\sum_{v\in V(\cG_0(\fp))} w(v)^{-1} & =|V(\cG_0(\fp))|-\frac{\kappa(\fp) q}{q+1}-1+\frac{(d-1)(3-2q)}{q-1}\\ \nonumber &=c_1 q^d +c_2' d+c_3',
\end{align}
\begin{equation}\label{align4.6}
\prod_{v\in V(\cG_0(\fp))}w(v) \left( \prod_{v\in E^+(\cG_0(\fp))}w(e)\right)^{-1}=(q-1)(q+1)^{\kappa(\fp)}=c_3'', 
\end{equation}
where $c_1, c_2, c_2', c_3'$ depend only $q$, and $c_3, c_3''$ depend on $q$ and the parity of $d$.

\begin{rem}\label{remDeg3}
The diagram in Example \ref{exampled3} is $\G_0(\fp)\bs\sT$ for $d=3$. 
\end{rem}


\subsection{Equidistribution of eigenvalues}\label{ssEE}
Let $\fp\lhd A$ be a prime ideal, and $\cG_0(\fp)$ be the finite part of the graph $\G_0(\fp)\bs \sT$ 
as in Definition \ref{defnG(p)}. With the degree of a vertex of $\cG_0(\fp)$ defined as in Definition \ref{defnDegVer},  
all vertices of $\cG_0(\fp)$, except the two boundary vertices, have degree $q+1$. 
The boundary vertices have degree $q$. (This is true for the boundary vertices of any $\cG_0(\fn)$ and follows from property (viii) of Ramanujan diagram.) 
In the notation of Lemma \ref{lem3.5} the boundary vertices of $\cG_0(\fp)$ are $v_1'$ and $v_{d-1}'$. 

Let $n(\fp)=|V(\cG_0(\fp))|$. 
Let $$0=\gamma_1(\fp)<\gamma_2(\fp)\leq \cdots\leq   \gamma_{n(\fp)}(\fp)$$ be the eigenvalues of $\Delta$ acting on $\cG_0(\fp)$. 
By Theorem \ref{thmMTFA} and Theorem \ref{thmRDD}, we have 
\begin{align*}
\gamma_2(\fp) &\geq q-2\sqrt{q}, \\ 
\gamma_{n(\fp)}(\fp) &\leq 2(q+1). 
\end{align*}
In this subsection we estimate the sum
$$
S(\fp):=\sum_{i=2}^{n(\fp)} \ln(\gamma_i(\fp))  
$$
as $\deg(\fp)\to \infty$. 

To simplify the notation we will sometime omit $\fp$ from notation, so, for example, $n$ in this paragraph is $n(\fp)$. 
Let $\delta$ be the adjacency operator on $\G_0(\fp)\bs \sT$, and $\delta'$ be the adjacency operator 
on $\cG_0(\fp)$. Let $\la_1, \dots, \la_n$ be the eigenvalues of $\delta'$. 
We have 
\begin{equation}\label{eqTrace}
\Delta=D-\delta'=((q+1)I-\delta')+(D-(q+1)I), 
\end{equation}
where $D$ is the diagonal matrix with the degrees of vertices of $\cG_0(\fp)$ on the diagonal (cf. the proof of Theorem \ref{thmMTFA}). 
Denote $\alpha_i:=(q+1)-\la_i$, $1\leq i\leq n$, the eigenvalues of the Hermitian matrix $(q+1)I-\delta'$. Without loss 
of generality, after reindexing, we assume $\alpha_1\leq  \cdots\leq \alpha_n$. By Theorem \ref{thmWeylHilb} and Theorem \ref{thmRDD}, we have
\begin{equation}\label{eqAlphaBound}
(\sqrt{q}-1)^2\leq \alpha_i\leq 2(q+1), \quad (2\leq i\leq n).
\end{equation}
Let $\beta_1\leq \cdots\leq \beta_n$ be the eigenvalues of the Hermitian matrix $D-(q+1)I$. 
Note that $\beta_1=\beta_2=-1$ and $\beta_3=\cdots=\beta_{n}=0$. By the Weyl's inequalities (Theorem \ref{thmWeylI}), 
we have 
$$
\alpha_i-1\leq \gamma_i\leq \alpha_i, \quad (1\leq i\leq n). 
$$
Hence we can write $\alpha_i=\gamma_i+\eps_i$, where $0\leq \eps_i\leq 1$. Taking the trace 
of both sides in \eqref{eqTrace}, we get 
$$
\sum_{i=1}^n \gamma_i = \sum_{i=1}^n \alpha_i +\sum_{i=1}^n \beta_i=\sum_{i=1}^n \alpha_i -2
$$
Thus, $\sum_{i=1}^n \eps_i=2$. This implies 
\begin{equation}\label{eqAlpha}
\sum_{i=2}^n\ln(\alpha_i)=\sum_{i=2}^n\ln(\gamma_i+\eps_i)= S(\fp)+c, 
\end{equation}
where $c$ is a constant which depends on $\fp$, but whose absolute value can be universally bounded 
independently of $\fp$. Thus, it is enough to estimate $\sum_{i=2}^n \ln(\alpha_i)$. 

Let $\{\nu_1, \dots, \nu_{m(\fp)}\}$ be the discrete spectrum of $\delta$. 
By Lemma \ref{lemCuspform},  $$\{\nu_1, \dots, \nu_m\}\subset \{\la_1, \dots, \la_n\}$$ 
and the eigenfunctions corresponding to $\nu_i$ are cusp forms. A formula for the dimension 
of the space spanned by cusp forms in $L_2(\G_0(\fp)\bs \sT)$ is given in \cite[Thm. 5.1]{HLW}. It follows from 
that formula that 
\begin{equation}\label{eqNq}
n(\fp)-m(\fp)\sim 2 \deg(\fp).   
\end{equation}
Since 
$$
n(\fp)\sim \frac{2|\fp|}{(q-1)^2(q+1)},  
$$
most of the eigenvalues of $\delta'$ come from cusp forms, although, as Example \ref{exampled3} demonstrates, 
the spectrum of $\delta'$ contains also values which are not in the spectrum of $\delta$. 
Combined with the bounds \eqref{eqAlphaBound}, this implies 
\begin{equation}\label{eqNu}
\sum_{i=2}^n\ln(\alpha_i) = \sum_{j=1}^m\ln((q+1)-\nu_j) + c' \deg(\fp),
\end{equation}
where $c'$ is a constant which depends on $\fp$, but whose absolute value 
can be universally bounded independently of $\fp$. Now 
we concentrate on estimating 
$$
S_\mathrm{cusp}(\fp):=\sum_{j=1}^m\ln((q+1)-\nu_j). 
$$
The key fact that we will use is the following:

\begin{thm}
Let $q$ be fixed. As $\deg(\fp)\to \infty$, the nontrivial discrete spectra $X_\fp:=\{\nu_1, \dots, \nu_{m(\fp)}\}$ of $\delta$ 
are equidistributed on $\Omega=[-2\sqrt{q}, 2\sqrt{q}]$ with respect to the measure 
$$
\mu_q(x)= \frac{q+1}{2\pi} \frac{\sqrt{4q-x^2}}{(q+1)^2-x^2}dx. 
$$
That is, for any continuous function $f$ on $\Omega$ the following holds:
$$
\lim_{\deg(\fp)\to \infty} \frac{1}{|X_\fp|}\sum_{\nu\in X_\fp} f(\nu) = \int_{\Omega} f(x)\mu_q(x). 
$$
\end{thm}
\begin{proof}
This is proven in \cite[Thm. 5.1]{Nagoshi}, following the method of Serre \cite{SerreJAMS} 
for cusp forms on the congruence subgroups of $\SL_2(\Z)$. 
\end{proof}

It follows from this theorem that the sets $X_\fp':=\{(q+1)-\nu_1, \dots, (q+1)-\nu_{m(\fp)}\}$ are 
equidistributed on $[(\sqrt{q}-1)^2, (\sqrt{q}+1)^2]$ with respect to the measure 
$$
\mu_q'(x)= \frac{q+1}{2\pi} \frac{\sqrt{4q-((q+1)-x)^2}}{(q+1)^2-((q+1)-x)^2}dx. 
$$

\begin{cor}\label{corCq}
As $\deg(\fp)\to \infty$, we have 
$$
S_\mathrm{cusp}(\fp)\sim m(\fp) C_q,
$$
where 
$$
C_q= \frac{q+1}{2\pi} \int_{(\sqrt{q}-1)^2}^{(\sqrt{q}+1)^2} \frac{\sqrt{4q-((q+1)-x)^2}}{(q+1)^2-((q+1)-x)^2}\ln(x) dx. 
$$
\end{cor}

The constant $C_q$ obviously depends only on $q$. 
Some of its approximate values, obtained with the help of computer program \texttt{SageMath}, 
are listed in Table \ref{tableCq}.  
\begin{table}
\begin{tabular}{ |c | c || c | c |  }
\hline
$q$ & $C_q$ & $q$ & $C_q$ \\
\hline
$2$ &  $0.837$ & $8$ &  $2.135$\\
\hline
$3$ &  $1.216$ & $9$ &  $2.247$\\
\hline
$4$ &  $1.483$ & $11$ &  $2.439$\\
\hline
$5$ &  $1.691$ & $13$ &  $2.601$\\
\hline
$7$ &  $2.008$ & $16$ &  $2.802$\\
\hline
\end{tabular}
\caption{}\label{tableCq}
\end{table} 
We have the following estimate: 
\begin{lem}\label{lemBobV}
$$
C_q=\ln\left(q+\frac{1}{2}\right)+O\left(q^{-2}\ln q\right). 
$$
\end{lem} 
\begin{proof}
Make the substitution $x=q+1-2\theta\sqrt{q}$ and use the symmetry of $\theta$ about $0$ to write 
$$
C_q=\frac{2(q+1)q}{\pi}\int_0^1\frac{\sqrt{1-\theta^2}}{(q+1)^2-4\theta^2 q}\ln((q+1)^2-4\theta^2 q)d\theta. 
$$
If we substitute the expansions 
$$
((q+1)^2-4\theta^2 q)^{-1}=\frac{1}{(q+1)^2}+\frac{4\theta^2 q}{(q+1)^4}+O(q^{-4})
$$
$$
\ln((q+1)^2-4\theta^2 q) = 2\ln(q+1)-\frac{4\theta^2 q}{(q+1)^2}+O(q^{-2}) 
$$
into the integral, and apply the formulae 
$$
\int_0^1\sqrt{1-\theta^2}d\theta=\frac{\pi}{4}, \qquad \int_0^1\theta^2\sqrt{1-\theta^2}d\theta=\frac{\pi}{16}, 
$$ 
we obtain 
$$
C_q=\frac{q}{q+1}\ln(q+1)+\frac{q^2}{(q+1)^3}(\ln(q+1)-1/2)+O(q^{-2}\ln(q+1)).
$$
The main terms here contribute 
$$
\frac{q^3+3q^2+q}{(q+1)^3}(\ln(q+1/2)+\ln(1+1/(2q+1)))-\frac{q^2}{2(q+1)^3}. 
$$
Expanding in powers of $q^{-2}$ gives the desired estimate.
\end{proof}

\begin{thm}\label{thmfDcG}
$$
\ln(\fD_{\cG_0(\fp), w}) \sim \frac{2 C_q}{(q-1)^2(q+1)}|\fp|. 
$$
\end{thm}
\begin{proof} 
Since $n(\fp)\sim m(\fp)\sim \frac{2|\fp|}{(q-1)^2(q+1)}$, combining Corollary \ref{corCq} with equations \eqref{eqAlpha} and 
\eqref{eqNu}, we get 
$$
S(\fp)\sim \frac{2 C_q}{(q-1)^2(q+1)}|\fp|. 
$$
Next, we rewrite the formula in Theorem \ref{thmMF} as 
$$
\fD_{\cG_0(\fp), w} \prod_{v\in V(\cG_0(\fp))}w(v) \left( \prod_{v\in E^+(\cG_0(\fp))}w(e)\right)^{-1} \sum_{v\in V(\cG_0(\fp))} w(v)^{-1} 
= \prod_{i=2}^{n(\fp)}\gamma_i(\fp). 
$$
Taking the logarithm of both sides and using \eqref{align4.5} and \eqref{align4.6}, we get
$$
\ln(\fD_{\cG_0(\fp), w}) \sim S(\fp). 
$$
\end{proof}


\subsection{Drinfeld modular curves: Proofs of main results}\label{ssDMC} Let $\fp\lhd A$ be a prime ideal. 
Denote $\F_\fp:=A/\fp\cong \F_{q^{\deg(\fp)}}$. Let $\F_\infty\cong \F_q$ be the residue field at $\infty$. 

\begin{thm}\label{thmDMC} There is a semi-stable curve $\cX_0(\fp)\to \p^1_{\F_q}$ such that: 
\begin{itemize}
\item[(i)] The generic fibre $\cX_0(\fp)_F$ is isomorphic to $X_0(\fp)$.  
\item[(ii)] $\cX_0(\fp)$ is smooth over $\Spec(A[\fp^{-1}])$.  
\item[(iii)] The dual graph of the special fibre $\cX_0(\fp)_{\F_\fp}$ at $\fp$ consists of two vertices joined by $s(\fp)$ edges, where 
$$
s(\fp)=\begin{cases}
\frac{|\fp|-1}{q^2-1} & \text{ if $\deg(\fp)$ is even},\\
\frac{|\fp|-q}{q^2-1}+1 & \text{ if $\deg(\fp)$ is odd}.
\end{cases}
$$
$($This graph looks like the graph in Example \ref{exBanana}.$)$ 
If $\deg(\fp)$ is even, then all edges have weight $1$. If $\deg(\fp)$ is odd, then one edge 
has weight $q+1$ and all other edges have weight $1$. 
\item[(iv)] The genus $g(\fp)$ of $X_0(\fp)$ is $s(\fp)-1$. 
\item[(v)] The dual graph of the special fibre $\cX_0(\fp)_{\F_\infty}$ at $\infty$ is the weighted graph $\cG_0(\fp)$. 
\end{itemize}
\end{thm}
\begin{proof}
(i) and (ii) follow from the results in \cite{Drinfeld} (see also \cite[Prop. V.3.5]{Lehmkuhl}); 
(iii) and (iv) follow from \cite[$\S$5]{Uber}; (v) follows from \cite[$\S$4.2]{PapikianAIF}. 
\end{proof}

\begin{proof}[Proof of Theorem \ref{thmMain1}] 
By Theorem \ref{thmGroth} and Theorem \ref{thmDMC} (v), 
$$
|\Phi_{J_0(\fp), \infty}| = \fD_{\cG_0(\fp), w}.$$ The estimate of Theorem \ref{thmMain1}
then follows from Theorem \ref{thmfDcG}. 
\end{proof}

\begin{proof}[Proof of Theorem \ref{thmMain2}] 

Let $\fD_\fp$ be the discriminant of $\G_0(\fp)\bs \sT$ defined in $\S$\ref{ssHCHO}. 
It is easy to see that $\fD_\fp=\fD_{\cG_0(\fp), w}$ since $H_1(\G_0(\fp)\bs \sT, \Z)=H_1(\cG_0(\fp), \Z)$ 
and the discriminants in question depend only on the cycles spanning the homology groups. 
The rank of $\cH_{0}(\fp, \Z)$ is equal to $g(\fp)$; cf. \cite[p. 49]{GR}.  

If the pairing  \eqref{GPairing} is perfect, then the bounds in Theorem \ref{thmMain2}
follow from Theorem \ref{thmfDT}, Theorem \ref{thmPIPest}, and Theorem \ref{thmfDcG}. 
On the other hand, it is easy to see from the proof of Theorem \ref{thmfDT} that 
the discriminant $\fD_{\T(\fp)}$ only increases if the pairing is not perfect.  
\end{proof}

Finally, we explain how to deduce the bounds on the height of the Jacobian $J_0(\fp)$ of $X_0(\fp)$ mentioned 
in the introduction. 

Let $\widetilde{\cX}_0(\fp)\to \cX_0(\fp)$ be the minimal desingularization. As follows from Theorem \ref{thmDMC} 
and Remark \ref{rem15}, the number of singular points $\varrho_\fp$ in the fibre of $\widetilde{\cX}_0(\fp)$ over $\fp$ 
is $s(\fp)$ if $\deg(\fp)$ is even, and $s(\fp)+q$ if $\deg(\fp)$ is odd. Similarly, the 
number of singular points in the fibre of $\widetilde{\cX}_0(\fp)$ over $\infty$ is 
$$
\varrho_\infty=\sum_{e\in E(\cG_0(\fp))^+} w(e). 
$$
By \eqref{eq33}, Lemma  \ref{lem3.4} and Lemma  \ref{lem3.5}, this last sum is equal to 
$$
\kappa(\fp)+(q+1)(|V_0|-1-\kappa(\fp))+\sum_{i=1}^{d-2}(q-1+|V_i|-2),
$$
Hence
$$
\varrho_\fp=g(\fp)+c, \quad \varrho_\infty= \frac{|\fp|}{(q-1)^2}+c'\deg(\fp)+c'',
$$ 
where $c, c', c''$ are constants depending only on $q$ and the parity of $\deg(\fp)$. 

\begin{thm}\label{lastTheorem}
$$
\frac{g(\fp)\deg(\fp)}{12}+o(g(\fp)\deg(\fp))\leq  H(J_0(\fp))\leq \frac{g(\fp)^2\deg(\fp)}{3}+o(g(\fp)^2\deg(\fp))
$$
\end{thm}
\begin{proof}
The bounds on the height $H(J_0(\fp))$ follow from Theorems \ref{thmH1}, \ref{thmH2}, 
and the previous estimates on $\varrho_\fp$ and $\varrho_\infty$. 
We only need to show that the inseparable exponent of $\cX_0(\fp)$ is $0$. If this is not the case, then $J_0(\fp)$ 
contains an abelian subvariety which is the Frobenius conjugate of another variety over $F$. This contradicts 
\cite[Thm. 1.1]{PapikianMA}.  
\end{proof}



\providecommand{\bysame}{\leavevmode\hbox to3em{\hrulefill}\thinspace}
\providecommand{\MR}{\relax\ifhmode\unskip\space\fi MR }
\providecommand{\MRhref}[2]{%
  \href{http://www.ams.org/mathscinet-getitem?mr=#1}{#2}
}
\providecommand{\href}[2]{#2}

\end{document}